\documentclass[letterpaper, 10 pt, conference]{ieeeconf}
\newcommand{\dist}{\text{dist}}
\IEEEoverridecommandlockouts
\usepackage{todonotes}
\usepackage{cite}
\usepackage{amsmath,amssymb,amsfonts}
\usepackage{algorithmic}
\usepackage{graphicx}
\usepackage{algorithm,algorithmic}
\usepackage{hyperref}
\hypersetup{hidelinks=true}
\usepackage{textcomp}
\newtheorem{lemma}{Lemma}
\newtheorem{definition}{Definition}
\newtheorem{theorem}{Theorem}
\newtheorem{proposition}{Proposition}
\newtheorem{assumption}{Assumption}
\newtheorem{remark}{Remark}

\newtheorem{corollary}{Corollary}
\usepackage{xcolor}
\usepackage[T1]{fontenc}
\usepackage{setspace}
\newcommand{\PLI}{P\L{}I}
\newcommand{\gPLI}{\textit{g}\PLI}

\newcommand{\kPLI}{$\mathcal{K}$-\PLI}
\newcommand{\kinfPLI}{$\mathcal{K}_\infty$-\PLI}
\newcommand{\kcPLI}{$\mathcal{K}_c$-\PLI}

\newcommand{\re}[1][\empty]{\mathbb{R}^{#1}}
\newcommand{\loss}{f}
\newcommand{\mloss}{\loss^{\ast}}

\usepackage[normalem]{ulem} 

\usepackage{cuted} 
\def\BibTeX{{\rm B\kern-.05em{\sc i\kern-.025em b}\kern-.08em
    T\kern-.1667em\lower.7ex\hbox{E}\kern-.125emX}}
\begin{document}
\title{\LARGE \bf Incremental Stability of Gradient Flows under Generalized Polyak--\L{}ojasiewicz Inequalities and Its Implications for General Systems}

\author{Andreas Oliveira, \IEEEmembership{Member, IEEE}, Arthur C. B. de Oliveira, \IEEEmembership{Member, IEEE}, \\ Mario Sznaier, \IEEEmembership{Fellow, IEEE} and Eduardo Sontag, \IEEEmembership{Fellow, IEEE}
\thanks{This work was partially supported by NSF grant CMMI 2208182, AFOSR grant FA9550-19-1-0005,  ONR grant N00014-21-1-2431 and NCCR Automation grant 51NF40\_225155 from the Swiss National Science Foundation. The authors are with the ECE Department at Northeastern University, Boston, MA 02115. \textit{(Corresponding author: Andreas Oliveira, e-mail: franciscodemelooli.a@northeastern.edu).}}}

\maketitle

\begin{abstract}
Gradient flows of twice continuously differentiable strongly convex and convex objectives are, respectively, incrementally exponentially stable and non-expansive in the Euclidean metric. In this work, we investigate to what extent these properties extend to nonconvex objectives satisfying generalized Polyak--\L{}ojasiewicz inequalities. For objectives with a unique minimizer, we establish conditions under which the gradient flow is contracting in a state-dependent Riemannian metric and incrementally exponentially stable. For objectives with nonunique minimizers, we generalize the non-expansivity of convex objectives to incremental stability for more general nonconvex objectives. We focus on incremental stability due to its implication to convergence to equilibrium points; when paired with a Lyapunov function it implies pointwise convergence to a minimizer and continuity of the endpoint map. Along the way, we derive new structural results for incrementally stable nonlinear systems and extend the analysis to Heavy-Ball and Dynamical Inertial Newton flows.
\end{abstract}

\begin{keywords}
gradient flows, \L{}ojasiewicz inequality, contractivity
\end{keywords}

\section{Introduction}\label{sec:introduction}

Continuous-time dynamical systems provide a useful framework for studying the qualitative behavior of optimization algorithms. The starting point for this approach is the gradient flow
\begin{equation}
\dot{x}=-\nabla f(x),
\end{equation}
for which the objective $f(x)$ is a natural Lyapunov function:
\begin{equation}
\frac{d}{dt}f(x(t))=-\|\nabla f(x(t))\|^2.
\end{equation}
Consequently, conditions relating the gradient norm to the optimality gap (i.e., $f(x) - f^{\ast}$) directly yield convergence guarantees for $f(x(t))-f^\ast$. Such guarantees, however, need not provide information about convergence of the state $x(t)$ to a particular minimizer of $f$. In particular, convergence of the objective value, convergence of the state to the minimizing set, and convergence of the state to a single minimizer are distinct properties. The same distinction arises for inertial optimization dynamics widely employed in descent methods to improve convergence rates, including the Heavy-Ball (HB) system \cite{polyak,attouch2000heavy} and the Dynamical Inertial Newton (DIN) system \cite{alvarez2002second}. Understanding when objective-value guarantees can be strengthened to state-space guarantees is the motivation of this work. 

Our main inspiration for the above questions comes from two important properties of convex objectives which make the situation easier to handle. The gradient flow of a convex $C^2$ function is non-expansive in the Euclidean norm (c.f. Remark \ref{rm:inf-contracting-convex})
\begin{equation}\label{eq:intro-non-exp}
\|x(t)-y(t)\| \leq \|x(0)-y(0)\|,
\end{equation}
while $m$-strong convexity gives incremental exponential stability (IES) with overshoot one:
\begin{equation}\label{eq:intro-IES}
\|x(t)-y(t)\|
\leq
e^{-mt}\|x(0)-y(0)\|.
\end{equation}
We prove generalizations of the two to nonconvex objectives with the goal of providing more tools to prove the convergence guarantees mentioned in the previous paragraph, thereby linking the study of optimization algorithms to stability concepts more prevalent in control theory akin to \cite{lessard2016analysis, fazlyab2018analysis, davydov2025time, Wensing_2020}.

A central relaxation of convexity is the Polyak--\L{}ojasiewicz inequality (\PLI), which requires, for some $\xi>0$,
\begin{equation}\label{eq:intro-PLI}
\|\nabla f(x)\|
\geq
\sqrt{\xi\big(f(x)-f^\ast\big)}.
\end{equation}
The \PLI\ does not require convexity, yet it preserves global optimality of stationary points and exponential convergence of the objective along gradient descent and gradient-flow trajectories \cite{karimi2016linear}. This property has proved useful in a broad range of nonconvex problems, including policy optimization for the linear-quadratic regulator (LQR) \cite{fazel2018lqr}, over-parameterized learning and neural networks \cite{liu2022loss}, low-rank matrix recovery \cite{bi2022local}, and nonconvex imaging problems \cite{farrell2022distributed}. It has also been used to study stability and generalization properties of learning algorithms beyond the strongly convex setting \cite{charles2018stability}.

The \PLI\ prescribes a particularly strict relationship between the gradient norm and the optimality gap. More general \L{}ojasiewicz-type conditions replace this square-root dependence by
\begin{equation}\label{eq:intro-generalized-pli}
\|\nabla f(x)\|
\geq
\alpha\big(f(x)-f^\ast\big),
\end{equation}
for a suitable positive-definite comparison function $\alpha$. Such generalizations arise naturally when a uniform global \PLI\ is too restrictive \cite{Arthur-PLI,cui2024small,cui2025small}. Of particular relevance to control, generalized gradient-dominance conditions have been motivated by continuous-time LQR policy optimization, where the available $\alpha(\cdot)$ that certifies \eqref{eq:intro-generalized-pli} is not of the square-root form in \eqref{eq:intro-PLI} and hence weaker than a global \PLI. 

Similarly to equation \eqref{eq:intro-PLI}, we have that \eqref{eq:intro-generalized-pli} directly constrains the evolution of the objective along the gradient flow through the same Lyapunov function argument. However, it does not, by itself, provide a certificate for the convergence, or rate of convergence, of the state to a stationary point, something that we can precisely infer for convex objectives from \eqref{eq:intro-IES} and \eqref{eq:intro-non-exp}. This motivates the main question of the paper: \emph{when do gradient flows satisfying generalized \L{}ojasiewicz inequalities possess incremental stability properties?} In the unique-minimizer setting, we first ask whether the contraction property enjoyed by strongly convex gradient flows can be extended to generalized \PLI\ objectives. In the nonunique-minimizer setting, strict contraction or IES are no longer possible. Motivated instead by the non-expansiveness of convex gradient flows, we ask whether an analogous weaker property can be recovered and study incremental stability as its natural generalization.

Contraction theory provides a natural tool for studying the unique-minimizer question, since strict contraction gives a sufficient certificate for IES. Classical logarithmic-norm criteria have been applied to nonlinear systems, chemical reaction networks, reaction-diffusion equations, robustness to inputs and time-varying disturbances, and entrainment \cite{contractive-with-inputs-sontag,Sontag-Aminzare,russo2010entrainment,oliveira2026time}. More general differential approaches based on Riemannian metrics \cite{LOHMILLER1998683} and Finsler Lyapunov functions \cite{forni2013differentiallyapunovframeworkcontraction} have found applications in observer design \cite{aghannan2003intrinsic}, nonlinear feedback synthesis \cite{manchester2017control}, geodesic optimization \cite{Wensing_2020}, learning and generalization \cite{kozachkov2022generalizationsupervisedlearningriemannian}, model-based and data-driven control \cite{manchester2017control, oliveira2025DDC}. For gradient flows, logarithmic-norm contraction and non-expansiveness are equivalent to strong convexity and convexity, respectively \cite{fitzsimmons2022note}; therefore most of our results for nonconvex generalized-\PLI\ objectives with unique minimizers will rely solely on Riemannian metrics. With nonunique minimizers, strict contraction and IES are precluded by the presence of distinct equilibria and we focus instead on incremental stability and its structural consequences for gradient flows and general nonlinear systems.

The main contributions of this paper are as follows:

\begin{enumerate}

\item We first slightly generalize our previous results presented in \cite{oliveira2026incremental}. We show that for a gradient flow under a generalized \kcPLI\ condition and local strong convexity of the unique global minimizer, we can construct a global Riemannian contraction metric whose induced distance is proper and induces the Euclidean topology, although it need not be uniformly equivalent to the Euclidean distance. Thus, the contraction property of strongly convex gradient flows extends to a broader class of nonconvex objectives under a state-dependent metric. 

\item We then identify two sufficient regimes under which this metric is uniformly equivalent to the Euclidean metric: one where the region of insufficient convexity is bounded, and another combining a global lower curvature bound with sufficient growth of the generalized gradient-dominance function. In either case, the gradient flow is globally incrementally exponentially stable in the Euclidean norm, generally with a nonunit overshoot. Compared to \cite{oliveira2026incremental} we also provide a new nonconvex example that satisfies the assumptions of the global lower curvature theorem.

\item We develop a novel complementary logarithmic-norm criterion for IES. In particular, we show that a globally asymptotically stable system that is infinitesimally contracting everywhere outside a compact set not containing the equilibrium is globally incrementally exponentially stable.

\item We study the structural consequences of incremental stability for general nonlinear systems. On a connected domain, incremental stability upgrades local asymptotic stability of a set to global asymptotic stability. For the equilibrium set, we characterize the existence of a proper compatible metric giving both incremental stability and local asymptotic stability by convergence of every trajectory to a Lyapunov-stable equilibrium, that is, global semistability \cite{bhat2003nontangency,bhat2010arclength}. The link between incremental stability and global semistability yields an alternative proof of the continuity of the endpoint map established for semistable systems in \cite{bhat2010arclength}. Our incremental formulation also provides a modulus of continuity in the chosen metric and a non-expansive metric constructed from the flow. The equilibrium set is consequently connected, and path-connected when the domain is path-connected.

\item We apply these structural results to gradient systems with nonunique minimizers. We show that gradient flows generated by classical \PLI\ objectives are incrementally stable. We also identify a new class of generalized-\PLI\ objectives whose gradient flows are also incrementally stable and which is neither weaker nor stronger than the classical \PLI\ class. This yields pointwise state convergence, continuity of the endpoint map, and path-connectedness of the minimizer set without requiring strict contraction or IES.

\item Finally, we extend the incremental-stability viewpoint beyond first-order gradient flows to inertial optimization dynamics. For convex objectives under the stated regularity assumptions (and, for DIN, also for classical \PLI\ objectives), we show that both the Heavy-Ball (HB) and Dynamical Inertial Newton (DIN) flows are incrementally stable. Combined with the general results above, this provides an alternative proof for continuity of their endpoint maps.

\end{enumerate}

\section{Preliminaries}

\subsection{Stability Notions}
Consider the nonlinear ordinary differential equation (ODE):
\begin{equation}\label{eq:ode}
    \dot{x} = g(x),
\end{equation}
where $g:\mathbb{R}^n \to \mathbb{R}^n$ is continuously differentiable. Throughout, $\varphi(x,t)$ denotes the solution of \eqref{eq:ode} at time $t$ from the initial condition $x$. We utilize the definitions of (global) exponential stability as given in \cite{khalil2002nonlinear}. As in \cite{khalil2002nonlinear} we utilize the notions of class-$\mathcal{K}$ and class-$\mathcal{K}_{\infty}$ functions:

\begin{definition}
A scalar continuous function $\alpha(r)$, defined for $r \in \mathbb{R}_{\ge 0}$, is said to belong to class $\mathcal{K}$ if it is strictly increasing and $\alpha(0)=0$. It is further said to belong to class $\mathcal{K}_{\infty}$ if $\alpha(r)\to\infty$ as $r\to\infty$.
\end{definition}

Following \cite{forni2013differentiallyapunovframeworkcontraction,angeli2002incremental} we define incrementally stable systems. We allow incremental stability to hold with respect to any proper metric $d$ (i.e. its closed balls are compact in the Euclidean topology) which induces the same topology as the Euclidean metric; we call such a metric \emph{compatible}.

\begin{definition}
System \eqref{eq:ode} is said to be \emph{incrementally exponentially stable} (IES) in a forward invariant set $G \subset \mathbb{R}^n$ w.r.t. a metric $d$ if there exists a rate $\kappa > 0$ and overshoot $c>0$ such that for any two initial conditions $x_0,y_0 \in G$ the corresponding solutions $\varphi(x_0,t)$, $\varphi(y_0,t)$ of \eqref{eq:ode} satisfy:
\begin{equation}
\label{eq:ies}
    d(\varphi(x_0,t),\varphi(y_0,t))
    \;\le\; c\,e^{-\kappa t}\,d(x_0,y_0),
    \, \forall\, t\ge 0.
\end{equation}
\end{definition}
Global IES of \eqref{eq:ode} w.r.t.\ the Euclidean metric on $\mathbb{R}^n$ implies the existence of a global exponentially stable equilibrium point, but the converse is false.
\begin{definition}
The system \eqref{eq:ode} is said to be \emph{incrementally stable} (IS) in a forward invariant set $G \subset \mathbb{R}^n$ with respect to a metric $d$ if there exists a class-$\mathcal{K}$ function $\sigma$ such that, for every $x_0,y_0 \in G$, and every $t \geq 0$, it holds that:
\begin{equation}\label{eq:I.S}
    d\!\left(
        \varphi(x_0,t),
        \varphi(y_0,t)
    \right)
    \leq
    \sigma\!\left(d(x_0,y_0)\right).
\end{equation}
\end{definition}

If $\sigma$ is the identity and $d$ is defined via any $\ell_p-$norm then we say the system is \emph{non-expansive}. 

As before let $G$ be forward invariant and let $S \subseteq G$ be nonempty and closed relative to $G$. For any point $x \in G$ its distance to $S$ is: 
\begin{equation}
    \dist(x,S) := \inf_{s\in S} d(x,s).
\end{equation}

\begin{definition} 
The set \(S\) is \emph{Lyapunov stable} relative to \(G\) if, for every \(\delta>0\), there exists \(\eta>0\) such that for every $x \in G$ where:
\begin{equation}
    \dist(x,S)< \eta \quad \Longrightarrow \quad \dist(\varphi(x,t),S) <\delta, \qquad \forall t\ge 0.
\end{equation}
The set $S$ is \emph{locally asymptotically stable} relative to $G$ if it is Lyapunov stable relative to $G$ and there exists $\rho>0$ such that for every $x \in G$:

\begin{equation}
    \dist(x,S)<\rho \quad \Longrightarrow \quad \dist(\varphi(x,t),S)\to 0 \quad \text{as } t\to\infty.
\end{equation} 

The basin of attraction of $S$ is 
\begin{equation}
    B_S := \big\{ x\in G: \dist(\varphi(x,t),S)\to 0 \text{ as } t\to\infty \big\}.
\end{equation} 

If $B_S=G$, then $S$ is said to be \emph{globally attractive} relative to \(G\). If, in addition, \(S\) is Lyapunov stable, then \(S\) is \emph{globally asymptotically stable} relative to \(G\). \end{definition}

\begin{definition}[Global semistability \cite{bhat2003nontangency,bhat2010arclength}]\label{def:semistability}
Let $G$ be forward invariant and let $S=\{x\in G:g(x)=0\}$ be its equilibrium set. The system is \emph{globally semistable relative to $G$} if the endpoint map
\begin{equation}
    P:G\to S,\qquad P(x):=\lim_{t\to\infty}\varphi(x,t),
\end{equation}
is well defined for every $x\in G$, and every equilibrium in $S$ is Lyapunov stable relative to $G$. Throughout the paper, $P$ always denotes the endpoint map of the system under consideration.
\end{definition}

\subsection{Gradient Flows and the Generalized Polyak--\L{}ojasiewicz Inequalities}

Throughout the text, let $\loss:\mathbb{R}^n\to\re$ be a twice continuously differentiable function. Assume $\loss$ is bounded below and that it attains its minimum at some point in $\mathbb{R}^n$, and consider the following optimization problem
\begin{equation}
        \label{eq:def-optprob}
        \underset{x\in\mathbb{R}^n}{\textrm{minimize}}\quad\quad f(x),
\end{equation}
to be solved by finding the equilibrium points of the gradient flow 
\begin{equation}\label{eq:gradient-flow}
    \dot x =-\nabla f(x).
\end{equation}
It is well known that if $\loss$ is coercive with a well-defined minimum value, then trajectories of the gradient flow will be pre-compact (see \cite{Arthur-PLI} for a brief overview). If we further assume that $\loss$ is real-analytic, then by \L{}ojasiewicz Theorem we know that every trajectory converges to a unique critical point of $f$ and that trajectories have finite length; finite length and convergence persist under more general Kurdyka-\L{}ojasiewicz inequalities \cite{kurdyka1998gradients,absil2005convergence,bolte2010characterizations}.

Although these results guarantee convergence of trajectories to the set of critical points, $\{x \in \mathbb{R}^n~|~\nabla\loss(x)=0\}$, there is no guarantee that they will converge to the set of global minimizers of $\loss$, $\{x\in \mathbb{R}^n ~|~f(x)=\min_{\mathbb{R}^n}f\}$. A common assumption in the nonconvex optimization literature is to require $\loss$ to satisfy a Polyak--\L{}ojasiewicz inequality, with the recent works \cite{Arthur-PLI,cui2025small,cui2024small} proposing a generalized framework for the standard inequality as follows:

\begin{definition}[Generalized Polyak--\L{}ojasiewicz inequalities]
\label{def:genPLI}
Let $\loss:\mathbb{R}^n\to\mathbb{R}$ and let
$\mloss:=\min_{x\in\mathbb{R}^n}\loss(x)$. We say that $\loss$ satisfies a
\emph{generalized Polyak--\L{}ojasiewicz inequality} if there exists a
continuous function $\alpha:\mathbb{R}_{+}\to\mathbb{R}_{+}$, with
$\alpha(r)=0$ if and only if $r=0$, such that
\begin{equation}
\label{eq:def-genPLI}
    \|\nabla \loss(x)\|
    \geq
    \alpha\big(\loss(x)-\mloss\big),
    \qquad \forall x\in\mathbb{R}^n.
\end{equation}
According to the properties of $\alpha$, we make the following
distinctions:
\begin{itemize}
    \item We refer to the generalized \PLI\ itself, with no further restriction on $\alpha$, as a \kcPLI. Since $\alpha$ is continuous and vanishes only at $r=0$, it is bounded away from zero on every compact subinterval of $(0,\infty)$:
    \begin{equation}
        \inf_{s\in[a,b]}\alpha(s)>0,
        \qquad
        \forall\,0<a<b<\infty.
    \end{equation}
    \item $\loss$ satisfies a \kPLI\ if $\alpha$ is of class $\mathcal{K}$.
    \item $\loss$ satisfies a \kinfPLI\ if $\alpha$ is of class $\mathcal{K}$ and
    \begin{equation}
        \lim_{r\to\infty}\alpha(r)=\infty.
    \end{equation}
\end{itemize}
\end{definition}

The classical Polyak--\L{}ojasiewicz inequality is recovered when
\begin{equation}
\label{eq:def-gPLI}
\alpha(r)=\sqrt{\xi r}
\end{equation}
for some $\xi>0$; we refer to this case as a \emph{global} \PLI\
(\gPLI).

\subsection{Contractivity and Weak Contractivity under Logarithmic Norms}

In this section, we introduce the class of logarithmic norms, define contraction with respect to such functionals, and show that this class of metrics is too restrictive to prove contractivity of gradient flows satisfying a \kinfPLI.

\begin{definition}{\cite{desoer2009feedback}}
Let $\|\cdot\|$ denote a norm on $\mathbb{R}^n$, and let the same symbol denote its induced matrix norm on $\mathbb{R}^{n \times n}$. The \emph{logarithmic norm} of a matrix $A \in \mathbb{R}^{n \times n}$ associated with $\|\cdot\|$ is defined by
\begin{equation}
    \mu(A) := \lim_{h \to 0^+} \frac{\|I + hA\| - 1}{h}.
\end{equation}
\end{definition}

Unless otherwise specified, $\|\cdot\|=\|\cdot\|_2$. We denote by $\mu_p$ the logarithmic norm induced by the $\|\cdot\|_p$, and by $\mu_{p,Q}$ the logarithmic norm induced by the weighted norm $\|x\|_{p,Q}:=\|Qx\|_p$.

Now consider the ODE \eqref{eq:ode}, with Jacobian denoted by $\nabla g := \frac{\partial g}{\partial x}$. The notion of contraction based on logarithmic norms and presented in \cite{contractive-with-inputs-sontag} is defined via the Jacobian $\nabla g$ and presented below:

\begin{definition}{\cite{contractive-with-inputs-sontag}}
The system \eqref{eq:ode} is said to be \emph{infinitesimally contracting} on a set $G \subset \mathbb{R}^n$ if there exists a norm $\|\cdot\|$, with associated logarithmic norm $\mu$, and a constant $\nu>0$ such that
\begin{equation}\label{def:infinitesimally-contracting}
    \mu(\nabla g(x)) \leq -\nu,
    \qquad \forall x\in G.
\end{equation}
If instead
\begin{equation}\label{def:infinitesimally-nonexpansive}
    \mu(\nabla g(x)) \leq 0,
    \qquad \forall x\in G,
\end{equation}
we say that the system is \emph{infinitesimally weakly contractive} on $G$ with respect to $\|\cdot\|$.
\end{definition}

\begin{remark}{\cite{contractive-with-inputs-sontag}}
Informally, the key consequence of infinitesimal contraction is that, whenever trajectories are forward invariant in a convex set $G$, the system is IES on $G$. Likewise, infinitesimal weak contraction implies the flow is non-expansive in the corresponding norm.
\end{remark}

Now we introduce an interesting connection between the logarithmic norm, the spectral abscissa and the induced matrix norm.

\begin{lemma}{\cite{desoer2009feedback}}\label{lemma:tightness-sysmmetric}
Let $A \in \mathbb{R}^{n \times n}$ and let $\|\cdot\|$ be any norm with associated logarithmic norm $\mu(\cdot)$. If $\Lambda(A) = \max\{\Re(\lambda) : \lambda \text{ is an eigenvalue of } A\}$ denotes the spectral abscissa, then, for every eigenvalue $\lambda$ of $A$,
\begin{equation}
-\|A\| \;\le\; -\mu(-A) \;\le\; \Re(\lambda) \;\le\; \Lambda(A) \;\le\; \mu(A) \;\le\; \|A\|.
\end{equation}

Further if $A \in \mathbb{R}^{n \times n}$ is symmetric then $\mu_2(A) = \Lambda(A) =  \lambda_{\max}(A)$.
\end{lemma}

This result is particularly useful in the context of this paper, since for gradient systems, the Jacobian of the vector-field is the negation of the Hessian of the cost, which is always symmetric. This allows one to prove the following statement about contractivity of gradient flows with respect to logarithmic norms:

\begin{remark}{\cite{fitzsimmons2022note}}\label{rm:inf-contracting-convex}
    The system \eqref{eq:gradient-flow} is infinitesimally contracting in $\mathbb{R}^n$ with contraction rate $\nu$ if and only if $f$ is $\nu-$strongly convex, and it is infinitesimally weakly contractive if and only if $f$ is
    convex.
\end{remark}

\begin{remark}\label{rm:no-log-norm-contraction}
    The gradient flow of a \gPLI\ objective can fail to be infinitesimally contracting, and even infinitesimally weakly contractive, with respect to \emph{every} norm. Indeed, a \gPLI\ objective need not be convex, so there can exist points $x$ where $\lambda_{\min}(\nabla^2 f(x))<0$. By Lemma~\ref{lemma:tightness-sysmmetric}, for any norm with associated logarithmic norm $\mu$,
    \begin{equation}
        \mu(-\nabla^2 f(x)) \ge \Lambda(-\nabla^2 f(x)) = -\lambda_{\min}(\nabla^2 f(x)) > 0.
    \end{equation}
\end{remark}

\textbf{Example:} Consider $f(x) = \frac{1}{2}x^2 - 2 \cos(x)$. This function can be verified numerically to satisfy a \gPLI\ with $\xi = \frac{1}{4}$ and has a unique critical point at the origin, but $\nabla^2 f(x) = 1 + 2\cos(x) < 0$ on a set of infinite measure.

\subsection{Contractivity and Weak Contractivity under Riemannian Metrics}

The limitations of logarithmic norms for proving contractivity of gradient flows when the cost satisfies a \gPLI\ motivate us to instead consider the class of Riemannian metrics.

\begin{definition}{\cite{lee2018introduction}}
Let $G$ be an open, path-connected subset of $\mathbb{R}^n$.  For our purposes, a \emph{Riemannian metric} is a function $M \in C^{1}(G,\mathbb{S}^n)$ such that $M(x)$ is positive definite for all $x \in G$.  For $v,w \in \mathbb{R}^n$, the bilinear form $v^T M(x) w$ defines an inner product at $x$.  

We define the orbital derivative of $M$ as:
\begin{equation}
\dot M(x) := \left.\frac{d}{dt} M(\varphi(x,t))\right|_{t=0}.
\end{equation}
If $M \in C^1(G,\mathbb{S}^n)$, then
\begin{equation}
\dot M_{ij}(x) = \nabla M_{ij}(x) \cdot g(x), \quad i,j=1,\dots,n.
\end{equation}
\end{definition}

\begin{definition}
Let $M$ be a Riemannian metric and let
\[
\begin{aligned}
\Gamma(x_0,x_1)
:=\Big\{\gamma:[0,1]\to G \ \text{piecewise }C^1 \ \big|\ 
&\ \gamma(0)=x_0,\\
&\ \gamma(1)=x_1\Big\}.
\end{aligned}
\]
The (Riemannian) distance between $x_0,x_1\in G$ induced by $M$ is
\begin{equation}\label{eq:riem_distance}
\begin{aligned}
d_M(x_0,x_1)
:=\inf_{\gamma\in\Gamma(x_0,x_1)} \int_{0}^{1} (\frac{\partial}{\partial s} \gamma(s)^{\top} M(\gamma(s))\frac{\partial}{\partial s}\gamma(s))^{\frac{1}{2}}ds.
\end{aligned}
\end{equation}
\end{definition}

\begin{definition}{\cite{GIESLConverse}}\label{def:contraction-riemannian}
Let $M$ be a Riemannian metric. For $v \in \mathbb{R}^n$, define
\begin{equation}
L_M(x;v) := \tfrac{1}{2} v^{\top} \big(M(x) \nabla g + \nabla g^{\top} M(x) + \dot M(x)\big) v.
\end{equation}
Then set $\mathcal{L}_M(x) := \max_{v^T M(x) v = 1} L_M(x;v).$
We call $M$ a \emph{contraction metric} on $G$ with rate $\nu > 0$ if
\begin{equation}\label{eq:contractivity-of-M}
\mathcal{L}_M(x) \le -\nu, \quad \forall x \in G.
\end{equation}

If $G = \mathbb{R}^n$ then we say $M$ is a \emph{global contraction metric}. If instead $\nu = 0$ so that
\begin{equation}\label{eq:weak-contractivity-of-M}
    \mathcal{L}_M(x) \leq 0,
    \qquad \forall x\in G,
\end{equation}
we say that $M$ is a \emph{weak contraction metric} on $G$, and that
the system is \emph{weakly contractive} with respect to $M$. This property is also referred to as \emph{semi-contraction} in \cite{Wensing_2020}.
\end{definition}

\begin{remark}\label{rm:lyapunov-equation}
Condition \eqref{eq:contractivity-of-M} is equivalent to the matrix inequality,
\begin{equation}
M(x) \nabla g + \nabla g^{\top} M(x) + \dot M(x) \preceq -2\nu M(x), \quad \forall x \in G,
\end{equation}
where $A \preceq B$ means $A-B$ is negative semidefinite. Similarly, weak contractivity is equivalent to
\begin{equation}
M(x)\nabla g+\nabla g^\top M(x)+\dot M(x)
\preceq 0,
\qquad \forall x\in G.
\end{equation} 
If $M(x) = Q \in \mathbb{S}^n$ is constant then by Lemma 2.16 in \cite{FB-CTDS} these conditions are respectively equivalent to $\mu_{2,Q^{\frac{1}{2}}}(\nabla g)\leq-\nu$ and $\mu_{2,Q^{\frac{1}{2}}}(\nabla g)\leq0$.
\end{remark}

\begin{remark}{\cite{aghannan2003intrinsic}}
    Much like the case of contractive systems with respect to logarithmic norms, proving $M$ is a contraction metric on a forward invariant open path-connected set $G$ shows that: 
    \begin{equation}
    \label{eq:ARestimate}
    \displaystyle d_{M}\big(\varphi(x_0,t),\varphi(x_1,t)\big) \;\le\; e^{-\nu t}\, d_{M}(x_0,x_1).
    \end{equation}
    If $M$ is instead a weak contraction metric, then
    \begin{equation}\label{eq:weak-ARestimate}
    d_M\big(\varphi(x_0,t),\varphi(x_1,t)\big) \leq d_M(x_0,x_1),
    \end{equation}
    so the flow is non-expansive with respect to the Riemannian distance $d_M$.
\end{remark}

In general, if $M$ is a global contraction metric, estimate \eqref{eq:ARestimate} and \eqref{eq:weak-ARestimate} do not allow for a direct translation into estimates in the $\ell_2$ norm. The following proposition shows when $\eqref{eq:ARestimate}$ can be translated into $\ell_2$ norm bounds.

\begin{proposition}\label{prop:ies}
    Assume there exist $\overline{\lambda},\underline{\lambda}> 0$ such that $\underline{\lambda} I \preceq M(x) \preceq \overline{\lambda} I$ for all $x \in \mathbb{R}^n$ and $M$ is a global (weak) contraction metric ($\nu=0$ in the weak case). Then for any $x_0,x_1 \in \mathbb{R}^n$:
    \begin{equation}\label{eq:ies-euclidean}
        \|\varphi(x_0,t) - \varphi(x_1,t)\|_2 \leq e^{-\nu t} \big(\frac{\overline{\lambda}}{\underline{\lambda}}\big)^{\frac{1}{2}} \|x_0-x_1\|_2.
    \end{equation}
\end{proposition}

\begin{proof}
We will first prove the lower bound on $d_{M}\big(x_0,x_1\big)$. From the definition of distance, and by the fact that $v^TM(x)v \geq \underline{\lambda}\|v\|_2^2$, we observe that
\begin{equation}
\begin{aligned}
    & d_{M}\big(x_0,x_1\big) = 
    \inf_{\gamma\in\Gamma(x_0,x_1)} \int_{0}^{1} (\frac{\partial}{\partial s} \gamma(s)^{\top} M(\gamma(s))\frac{\partial}{\partial s}\gamma(s))^{\frac{1}{2}}ds \\
    & \geq \inf_{\gamma\in\Gamma(x_0,x_1)} \underline{\lambda}^{\frac{1}{2}} \int_{0}^1 (\frac{\partial}{\partial s} \gamma(s)^T\frac{\partial}{\partial s}\gamma(s))^{\frac{1}{2}} ds
     = \underline{\lambda}^{\frac{1}{2}} \|x_0 - x_1\|_2.
\end{aligned}
\end{equation}
where the last equality is obtained by the fact that for $M$ a constant matrix, the geodesic is a straight line. Similarly $d_{M}\big(x_0,x_1\big) \leq \overline{\lambda}^{\frac{1}{2}} \|x_0 - x_1\|_2$ and combining the two inequalities with \eqref{eq:ARestimate} proves \eqref{eq:ies-euclidean}.
\end{proof}

Thus, establishing IES for a gradient flow reduces to constructing a global contraction metric $M$ that is \emph{uniformly bounded above and below}. Lastly, we present an important lemma that will be used later to construct our Riemannian metrics. We omit the proof and defer the reader to Lemma 2.7 of \cite{GIESLConverse}.

\begin{lemma}{\cite{GIESLConverse}}\label{lemma:riemmanina-metric-with-lyapunov}
Let $N: G \to \mathbb{S}^n$ be a Riemannian metric and $V \in C^1(G,\mathbb{R})$.  
Then
\begin{equation}
M(x) := e^{2V(x)} N(x)
\end{equation}
is also a Riemannian metric, and
\begin{equation}
\mathcal{L}_M(x) = \mathcal{L}_N(x) + \dot V(x).
\end{equation}
\end{lemma}


The following proposition shows that the distance induced by a uniformly lower-bounded metric, such as the one constructed in Theorem~\ref{prop:global} below, is proper and compatible; it is also used in Section~IV.

\begin{proposition}\label{prop:equivalence-of-topologies}
If there exists $\underline{\lambda}>0$ such that
\begin{equation}
    \underline{\lambda} I \preceq M(x),\qquad \forall x\in\mathbb R^n,
\end{equation}
and $M\in C^1(\mathbb R^n,\mathbb S^n)$, then the Riemannian distance $d_M$ defined by \eqref{eq:riem_distance} induces the same topology as the Euclidean metric on $\mathbb R^n$. Moreover, $d_M$ is proper.
\end{proposition}

\begin{proof}
By the proof of Proposition \ref{prop:ies},
\begin{equation}
    d_M(x,y)\ge \underline{\lambda}^{\frac{1}{2}}\,\|x-y\|_2.
\end{equation}
Consequently, if \(d_M(x_k,x)\to0\), then
\begin{equation}
    \|x_k-x\|_2
    \le
    \frac{1}{\underline{\lambda}^{\frac{1}{2}}}d_M(x_k,x)
    \to0.
\end{equation}

Conversely, fix \(x\in\mathbb R^n\). Since $M$ is continuous, for every $r>0$ there exists $\overline{\lambda}_{x,r}>0$ such that
\begin{equation}
    M(z)\preceq \overline{\lambda}_{x,r} I,
    \qquad \forall z\in \overline B_r(x).
\end{equation}
If $y\in B_r(x)$, then the straight-line path
\begin{equation}
    \gamma(s)=x+s(y-x),\qquad s\in[0,1],
\end{equation}
is contained in $\overline B_r(x)$, hence by the proof of Proposition \ref{prop:ies}
\begin{equation}
\begin{aligned}
    d_M(x,y)
    &\le
    \overline{\lambda}_{x,r}^{\frac{1}{2}}\|x-y\|_2.
\end{aligned}
\end{equation}
Therefore, if $x_k\to x$ in the Euclidean norm, then for all sufficiently large $k$, $x_k\in B_r(x)$, and the above estimate gives
\begin{equation}
    d_M(x_k,x)\to0.
\end{equation}
Thus \(d_M\) and the Euclidean metric induce the same topology on $\mathbb R^n$.

It remains to prove properness. Fix $p\in\mathbb R^n$ and $R>0$. From
the uniform lower bound on $M$,
\begin{equation}
    d_M(x,p)\ge \underline{\lambda}^{\frac{1}{2}}\|x-p\|_2.
\end{equation}
Therefore

\begin{equation}
\begin{aligned}
    & \{x \;\; | \;\; x \in \mathbb{R}^n \text{ and } d_M(x,p) \le R\} = \overline B^{d_M}_R(p)
    \subseteq \\
    & \overline B_{\frac{R}{\sqrt{\underline{\lambda}}}}(p) = \{x \;\; | \;\; x \in \mathbb{R}^n \text{ and } \|x - p\|_2 \le \tfrac{R}{\sqrt{\underline{\lambda}}}\}.
\end{aligned}
\end{equation}

Since $d_M$ induces the Euclidean topology, the set $\overline B^{d_M}_R(p)$ is closed in the Euclidean topology. It is also bounded in the Euclidean topology by the inclusion above and therefore it is compact. Thus $d_M$ is proper.
\end{proof}

\section{Unique equilibrium}

In this section, we consider the optimization problem in \eqref{eq:def-optprob} under the assumption that the cost $\loss$ is coercive, twice continuously differentiable, admits a unique global minimizer, is locally strongly convex around that minimizer, and satisfies a generalized \PLI.

We first show that, if $\loss$ satisfies a \kcPLI\ and is locally strongly convex around its minimizer, then there exists a proper Riemannian metric whose induced distance generates the Euclidean topology and under which the gradient flow is globally contractive, hence satisfying \eqref{eq:ARestimate}. Therefore the contraction property enjoyed by strongly convex gradient flows persists for a substantially broader class of nonconvex objectives after an appropriate change of metric. The metric obtained in this construction is proper and induces the Euclidean topology but it need not be uniformly equivalent to the Euclidean metric. The obstruction is that it may grow without a uniform upper bound, allowing the Riemannian distance between points to become arbitrarily large while maintaining the same Euclidean distance. We therefore derive two sufficient conditions, combining curvature restrictions on $\loss$ with growth properties of its generalized gradient-dominance bound, under which the contraction metric is uniformly bounded above and below. Under these conditions, \eqref{eq:ARestimate} can be converted into the Euclidean estimate \eqref{eq:ies-euclidean}, yielding global IES with a generally nonunit overshoot.

The existence of a global contraction metric for objectives satisfying a \kcPLI\ that need not reduce to a classical \PLI\ may initially appear counterintuitive, since the natural Lyapunov function $\loss-\mloss$ need not decay exponentially. This is consistent with the converse result in \cite{forni2015differential}, which establishes the existence of contraction metrics on the basin of attraction of a stable hyperbolic equilibrium. In contrast to the converse construction there which relies on the Hartman--Grobman Theorem, the metrics developed here are obtained explicitly from properties of the objective function.

\subsection{Contractivity of nonconvex gradient flows}

\begin{assumption}\label{assump:1}
The function $\loss$ is coercive and admits a unique global minimizer $x^{\ast}$ where it is strongly convex, hence there exists $\nu > 0$ such that $\lambda_{\min}(\nabla^2 f(x^{\ast})) \ge \nu$. Further, $f$ satisfies a \kcPLI.
\end{assumption}

\begin{theorem}\label{prop:global}
Suppose $\loss$ satisfies Assumption \ref{assump:1}. Then for any $\epsilon \in (0,\nu)$ there exists a uniformly lower bounded Riemannian metric $\hat{M} \in C^{1}(\mathbb{R}^n, \mathbb{S}^{n})$ such that
\begin{equation}\label{eq:globalness}
\begin{aligned}
    \mathcal{L}_{\hat{M}}(x) \;\le\; -\nu + \epsilon, \; \forall x \in \mathbb{R}^n.
\end{aligned}
\end{equation}
Hence, the system is \emph{contractive} with rate $\nu-\epsilon$.
\end{theorem}
The proof is provided in the appendix; however, we discuss the idea next. Notice that locally around $x^*$, Lemma \ref{lemma:tightness-sysmmetric}, Remark \ref{rm:inf-contracting-convex} and Remark \ref{rm:lyapunov-equation}, show that the metric $M(x) = I$ suffices to produce $\mathcal{L}_{I}(x) \le -\nu + \epsilon$ for $x$ in a sufficiently small neighborhood of $x^{\ast}$ and in fact it is uniformly negative around any point where $\loss$ is locally strongly convex. Further, inspired by the application of Lemma \ref{lemma:riemmanina-metric-with-lyapunov} in \cite{GIESLConverse}, given that the cost yields a natural Lyapunov function for the gradient flow, we can expect to construct a metric that satisfies \eqref{eq:globalness} globally. One challenge lies in interpolating between a local constant metric $I$, used where $\loss$ is strongly convex, and a state-dependent metric $M(x) = e^{2v(f(x))}I$ where $\loss$ isn't. 

\subsection{Incremental Exponential Stability of nonconvex gradient flows}

We now show how to construct uniformly bounded metrics that guarantee global contractivity under various \PLI\ bounds and different degrees of nonconvexity. 

\begin{theorem}\label{thm:concavity-on-annulus}
    Suppose $f$ satisfies Assumption \ref{assump:1}. Further assume $f$ is $m$-strongly-convex outside a compact set $K$ for some $m > 0$. Then the gradient flow is \emph{globally incrementally exponentially stable} and the equilibrium $x^{\ast}$ is globally exponentially stable.
\end{theorem}

The proof is provided in the appendix, but the main idea is to construct a globally uniformly bounded contraction metric. The key additional point beyond Theorem~\ref{prop:global} is that the set where $\loss$ fails to be strongly convex is bounded. Hence the contraction metric only needs to be state-dependent on a bounded region; outside it, one smoothly patches it to a constant metric while preserving the contraction inequality. As a result, the metric is uniformly bounded above and below, and Proposition~\ref{prop:ies} yields global IES. Figure~\ref{fig:annulus_concavity} illustrates this regime.

\begin{figure}[t]
    \centering
    \includegraphics[width=\columnwidth]{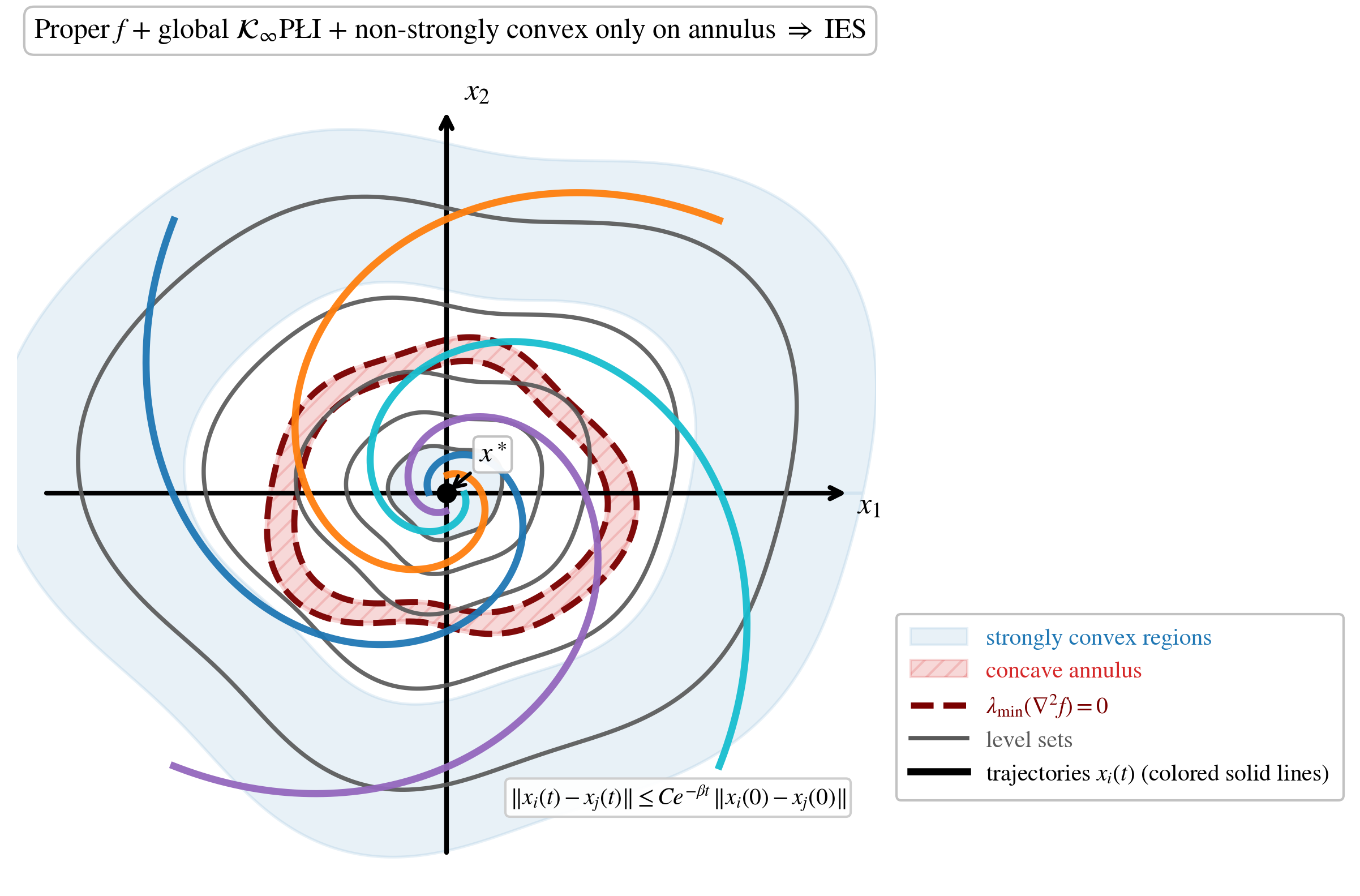}
    \caption{Schematic illustration of the assumptions and conclusion of Theorem~\ref{thm:concavity-on-annulus}. A coercive objective $f$ satisfying a \kcPLI. The equilibrium $x^\ast$ lies in a strongly convex neighborhood, and everything outside an annulus is strongly convex. A bounded region where $f$ is nonconvex, $\lambda_{\min}(\nabla^2 f(x))<  0$, is allowed and is shown in red. Gradient flow trajectories $\dot{x}=-\nabla f(x)$ starting in different quadrants tend exponentially toward each other since the system is IES by Theorem \ref{thm:concavity-on-annulus}.}
    \label{fig:annulus_concavity}
\end{figure}

Next, we look at proving IES for weakly convex functions:

\begin{definition}
    A function $f: \mathbb{R}^n \to \mathbb{R}$ is $m$-\emph{weakly convex} for some $m>0$ if $\lambda_{\min}(\nabla^2\loss(x))\ge-m$ for all $x\in\mathbb{R}^n$, i.e., if $f + \frac{m}{2}\|x\|^2$ is convex.
\end{definition}

The following class of generalized \PLI\ is useful:

\begin{definition}
    A function $f$ is said to satisfy a \emph{strong}-\PLI\ if it satisfies \eqref{eq:def-genPLI} for some comparison function $\alpha$ that is of class $\mathcal{K}_\infty$ and for any $s_0 > \mloss$
    \begin{equation}\label{eq:strong-PLI}
        \int_{s_0}^\infty \frac{1}{\alpha(s - f^{\ast})^{2}} \; ds<\infty.
    \end{equation}
\end{definition}
Condition \eqref{eq:strong-PLI} constrains $\alpha$ only at infinity, and it is not implied by the classical \gPLI, since $\int_{s_0}^{\infty} \frac{ds}{\xi(s - \mloss)} = \infty$. It holds, for example, when $\alpha(s) = cs^{p}$ with $c> 0$ and $p > \frac{1}{2}$. From this, we present the following result: 

\begin{theorem}\label{prop:above-PL-regime}
Suppose $f$ satisfies Assumption \ref{assump:1} and a strong-\PLI. Further assume $f$ is $m$-weakly convex for some $m>0$. Then the gradient flow is \emph{globally incrementally exponentially stable} and the equilibrium $x^{\ast}$ is globally exponentially stable.
\end{theorem}

The proof is given in the appendix and follows the same construction as in Theorem~\ref{prop:global}. The new issue is that $f$ may remain nonconvex on unbounded sets, so the metric must stay state-dependent at infinity. The strong-\PLI\ assumption is used to control this growth and ensure that the resulting metric remains uniformly bounded. Proposition~\ref{prop:ies} then implies global IES. Figure~\ref{fig:prop_integrability_envelope} illustrates this case.

\begin{figure}[t]
    \centering
    \includegraphics[width=\columnwidth]{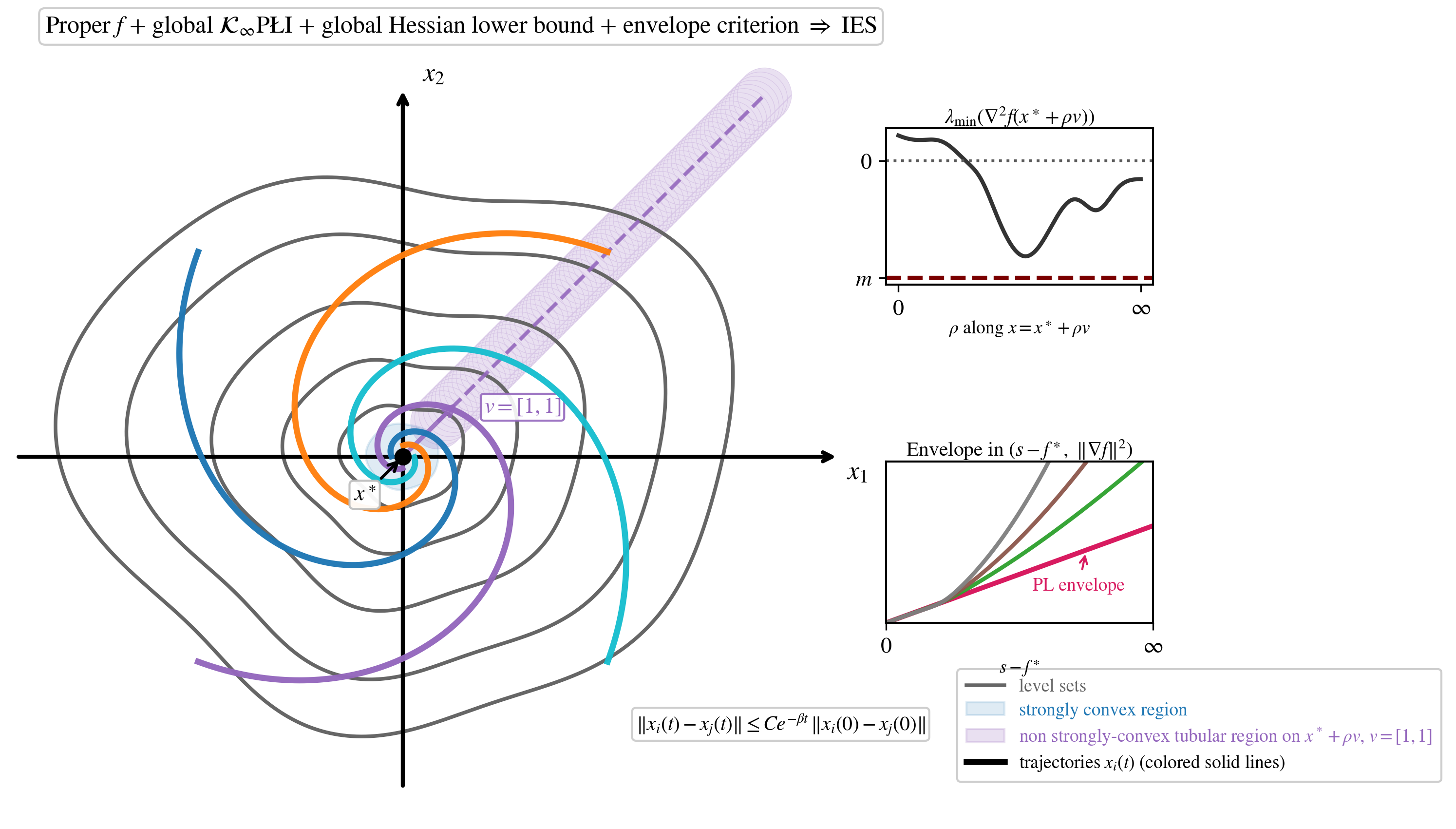}
    \caption{Schematic illustration of the assumptions and conclusion of Theorem \ref{prop:above-PL-regime}. The equilibrium $x^\ast$ lies in a strongly convex neighborhood (shaded blue region). A coercive objective $f$ satisfying a strong-\PLI, a growth condition at infinity that the classical \gPLI\ does not satisfy, is shown in the lower right graph, and a global lower bound on $\lambda_{\min}(\nabla^2 f)$ (which need not be positive) is shown in the top right graph. Gradient flow trajectories $\dot{x}=-\nabla f(x)$ starting in different quadrants tend exponentially toward each other since the system is IES by Theorem \ref{prop:above-PL-regime}.}
    \label{fig:prop_integrability_envelope}
\end{figure}

\subsection{A Generalization of Theorem \ref{thm:concavity-on-annulus} via Logarithmic Norms}

As noted in Remark \ref{rm:no-log-norm-contraction}, it is generally impossible to prove that the gradient flow of a \gPLI\ objective is infinitesimally contracting. However, the setting of Theorem \ref{thm:concavity-on-annulus} is more structured: the system is infinitesimally contracting everywhere except on a compact set that does not contain the global minimizer, and all trajectories converge asymptotically to that minimizer. This makes it possible to use logarithmic norms to show that the system is IES even though it is not infinitesimally contracting, an idea closely related to contraction after small transients \cite{margaliot2016transients}. One can generalize this beyond gradient systems:

\begin{theorem}\label{thm:ies-inf-contracting-finite-time}
Suppose there exists a set $K\subset\mathbb{R}^n$, a norm $\|\cdot\|$ with associated logarithmic norm $\mu$, and constants $\underline{\nu}>0$, $\overline{\nu}\geq-\underline{\nu}$, and $T>0$ such that
\begin{equation}
    \mu(\nabla g(x))\leq-\underline{\nu},
    \qquad \forall x\in K,
\end{equation}
and
\begin{equation}
    \mu(\nabla g(x))\leq \overline{\nu},
    \qquad \forall x\in K^c.
\end{equation}
Further assume that every trajectory of \eqref{eq:ode} spends at most a total time $T$ in $K^c$. Then \eqref{eq:ode} is incrementally exponentially stable with respect to $\|\cdot\|$, with rate $\underline{\nu}$ and overshoot $e^{(\overline{\nu} + \underline{\nu})T}$.
\end{theorem}

As a Corollary one is able to prove a result akin to Theorem \ref{thm:concavity-on-annulus} but for general systems:

\begin{corollary}\label{cor:ies-inf-contracting-compact}
Suppose \eqref{eq:ode} has a unique globally asymptotically stable equilibrium $x^\ast$ and is infinitesimally contracting, with rate $\nu$ and w.r.t. a norm $\|\cdot\|$, everywhere outside a compact set $C\subset\mathbb{R}^n$ with $x^\ast \notin C$. Then \eqref{eq:ode} is incrementally exponentially stable.
\end{corollary}

We refer the reader to the Appendix for the proof.

\subsection{Example of an objective satisfying a \kinfPLI\ condition but not a classical one}
Consider the function:
\begin{equation}
    f(x) = x\log(x + \sqrt{1 + x^2}) - \sqrt{1 + x^2} + 1.
\end{equation} Since $ \log(x + \sqrt{1 + x^2}),\sqrt{1 + x^2}$ are $C^{\infty}$, $f$ is $C^{\infty}$. Computing:

\begin{equation}
    \nabla f(x) = \log(x + \sqrt{1 + x^2}), \quad \nabla^2 f(x) = \frac{1}{\sqrt{1 + x^2}}.
\end{equation}

we observe that the global minimizer is at $x = 0$ and the function is locally strongly convex around $0$. Now we find its global \kinfPLI\ bound. Assume $f$ satisfies a \kinfPLI\ with $\alpha(s) = (cs)^{\frac{1}{p}}$, then:

\begin{equation}
    \frac{|\nabla f(x)|^p}{f(x)} = \frac{[\log(x + \sqrt{1 + x^2})]^p}{f(x)} \ge c.
\end{equation}

In the limit as $x \to \infty$, $\nabla f(x) \sim \log(2 x)$ and $f(x) \sim x\log(2x)$, hence we observe that for any $p > 0$, 

\begin{equation}
\lim_{x\to \infty} \frac{\log(x + \sqrt{1 + x^2})^p}{f(x)} \to 0.
\end{equation}

This proves the function $f$ satisfies no classical power-type \L{}ojasiewicz inequality, i.e $\alpha(s)$  cannot be of the form $(cs)^{\frac{1}{p}}$. Instead we will show that $f$ satisfies a \kinfPLI\ with  $\alpha(s) = \frac{1}{2}(\log(1 + s))$. First notice that $\log(x + \sqrt{1 + x^2})$ is an odd function  and $f(x)$ is an even function. Hence it suffices to show that for $x > 0$,

\begin{equation}\label{eq:log-K-inf}
    \nabla f(x) \geq \frac{1}{2}\log(1 + f(x)) \iff e^{2\nabla f(x)} \ge 1 + f(x).
\end{equation}

Using that $\log(x + \sqrt{1 + x^2}) \le x$ for $x > 0$ allows us to obtain that $f(x) \le x^2$  for $x > 0$ and we get that $1 + f(x) \le 1 + x^2$. Furthermore, $e^{2\nabla f(x)} = (x + \sqrt{1 + x^2})^2 \ge 1 + x^2$, which combining with \eqref{eq:log-K-inf} shows that $\alpha(s) = \frac{1}{2}(\log(1 + s))$ is clearly a \kinfPLI\ for the function $f$. Lastly, since $f$ is coercive and strongly convex around the minimizer $x^{\ast} = 0$ it satisfies the assumption of Theorem \ref{prop:global} and the gradient flow is contractive under the metric constructed there. We note that this $f$ is convex and the role of this example is to illustrate a comparison function $\alpha$ that is not of power type, while the next section provides a nonconvex example.

\subsection{Example of a nonconvex objective satisfying the assumptions of Theorem \ref{prop:above-PL-regime}}\label{sec:example-thm3}

Consider the coercive $C^{\infty}$ function:

\begin{equation}
    f(x) = \frac{x^4}{4} - \frac{x^2}{2} + \frac{3}{4}\sin(x^2).
\end{equation}

Its derivative equals

\begin{equation}
\begin{aligned}
    & f'(x) = x^3 - x + \frac{3}{2}x\cos(x^2) = xq(x^2), \\
    & \text{ where } q(t) = \left(t - 1 + \frac{3}{2}\cos(t)\right).
\end{aligned}
\end{equation}

\textit{Unique critical point:} For $t \ge 0$ we show that $q(t) > 0$. First we compute $q'$

\begin{equation}
    q'(t) = 1 - \frac{3}{2}\sin(t),
\end{equation}
and we notice that the critical points of $q$ satisfy
\begin{equation}
\begin{aligned}
    & \sin t = \frac{2}{3} \Rightarrow t= \arcsin\left(\frac{2}{3}\right) + 2k\pi,\\
    & \quad \qquad \text{ or } t =\pi - \arcsin\left(\frac{2}{3}\right) + 2k\pi.
\end{aligned}    
\end{equation}

Let $\theta_0 = \arcsin(\frac{2}{3})$. At $t = \theta_0 + 2k\pi$ we have that $\cos(t) > 0$, therefore $q''(t) = -\frac{3}{2}\cos(t) < 0$ proves these points are local maxima. For $t_k = \pi - \theta_0 + 2k\pi$ we have $\cos(t_k) = -\frac{\sqrt{5}}{3},$ so $q''(t_k) = \frac{\sqrt{5}}{2}$ and therefore these are local minima of $q$. Nonetheless, $q(\pi - \theta_0) \approx 0.294$, therefore, leveraging that at $t_k = \pi - \arcsin\left(\frac{2}{3}\right) + 2k\pi$,  $q(t_k) = q(\pi - \arcsin\left(\frac{2}{3}\right)) + 2k\pi$, we conclude $q(t_k)>0$. Since $q(0) = \frac{1}{2}$ and the local minima satisfy $q(t) > 0.29$ we conclude that the local maxima also satisfy $q(t)> 0.29$  and therefore $q(t)> 0.29$ for $t \ge 0$. Recalling that $f'(x) = xq(x^2)$ allows us to conclude that $x = 0$ is the unique stationary point of $f$.

\textit{Local strong convexity and global weak convexity:} Noting that

\begin{equation}
    f''(x) = 3x^2 - 1 - 3x^2\sin(x^2) + \frac{3}{2}\cos(x^2)
\end{equation}
we observe that  $f''(0) =  - 1 + \frac{3}{2} = \frac{1}{2}$, hence $x = 0$ is a global minimizer where $f$ is locally strongly convex. Furthermore,

\begin{equation}
\begin{aligned}
    & f''(x) = 3x^2 - 1 - 3x^2\sin(x^2) + \frac{3}{2}\cos(x^2) \\
    & = 3x^2\left(1 - \sin(x^2)\right) - 1 + \frac{3}{2}\cos(x^2),
\end{aligned}
\end{equation}
where the first term above is clearly greater than or equal to zero and the remainder is greater than or equal to $-\frac{5}{2}$, which implies the objective is $\frac{5}{2}$-weakly convex. At $x = \sqrt{2k\pi + \frac{\pi}{2}}$, $f''(x) = -1$, so $f$ is nonconvex on an unbounded set. 

\textit{Strong \PLI:} We seek to show that $|\nabla f(x)| \ge \beta(f(x) - 0)^{\frac{3}{4}}$ for some $\beta > 0$. This inequality, replacing $\beta$ with $\beta_u > 0$ can be rewritten as

\begin{equation}
\begin{aligned}
    & \frac{|\nabla f(x)|}{(\frac{x^4}{4} - \frac{x^2}{2} + \frac{3}{4}\sin(x^2))^{\frac{3}{4}}} = \frac{|x|q(x^2)}{(\frac{x^4}{4} - \frac{x^2}{2} + \frac{3}{4}\sin(x^2))^{\frac{3}{4}}} \ge \beta_u.
\end{aligned}
\end{equation}

In the limit as $|x| \to \infty$, the left-hand side converges to $4^{\frac{3}{4}} = 2\sqrt{2}$, which implies there exists $N$ such that for $|x| > N$ the above holds for any $\beta_u < 2\sqrt{2}$. Locally around the global minimizer $x^{\ast}$, $f$ is strongly convex hence there exists sufficiently small $N_{\ast} \in (0,1)$ such that for $|x|<N_{\ast}$, $f$ satisfies a \PLI. A \PLI\ near the stationary point implies $f$ satisfies a local \kcPLI\ in the same region $|x|<N_{\ast}$ and with $\alpha(s) = \beta_l(s)^{\frac{3}{4}}$ for some $\beta_l > 0$. Neither $f(x)-f(0)$ nor $|\nabla f(x)|$ equals zero in the compact region $N_{\ast} \le |x| \le N$,  hence defining

\begin{equation}
    \beta_m = \frac{\min_{N_{\ast} \le |x| \le N} |\nabla f(x) |}{\max_{N_{\ast} \le |x| \le N} {(f(x) - f(0))^{\frac{3}{4}}}} > 0,
\end{equation}
implies for $N_{\ast} \le |x| \le N$, $|\nabla f(x)| \ge \beta_m(f(x) - 0)^{\frac{3}{4}}$ holds. Choosing $\beta = \min\{\beta_l, \beta_m, \beta_u\}$ proves that $\alpha(s) = \beta(s)^{\frac{3}{4}}$ is a \kcPLI\ for $f$, which clearly also satisfies the strong-\PLI\ condition.

\section{Multiple Equilibria}

Objectives satisfying a \gPLI\ need not admit a unique minimizer, and one of the strengths of the Lyapunov argument is precisely that it extends to sets of minimizers. In this setting, however, the contraction constructions of the previous section cannot be applied directly. Indeed, a global contraction metric induces a contraction of the flow with respect to the associated distance and, by the Banach Contraction Mapping Principle, is incompatible with the existence of multiple equilibria \cite{Aylward_2008}. Likewise, incremental global asymptotic stability, in the sense of \cite{forni2013differentiallyapunovframeworkcontraction,angeli2002incremental}, cannot hold in the presence of distinct equilibria.

We therefore turn to the weaker notion of incremental stability, which is compatible with nonunique equilibria while still imposing strong restrictions on the asymptotic behavior of the flow. Under the assumptions introduced below, we characterize global semistability, as stated in Definition \ref{def:semistability}, by the existence of a proper compatible metric under which the system is incrementally stable and the equilibrium set is locally asymptotically stable. This characterization provides an alternative proof of the continuity of the endpoint map established for semistable systems in \cite{bhat2010arclength}. The incremental estimate also yields a modulus of continuity in the chosen metric. These properties capture the asymptotic behavior exhibited by \gPLI\ gradient flows and provide a natural framework for extending such conclusions to broader classes of objectives.

To motivate the need for additional incremental structure, we begin with an example showing that generalized gradient dominance alone is insufficient to guarantee convergence of the state to the minimizing set. Specifically, we construct a gradient flow satisfying a global \kinfPLI\ condition for which the objective converges to its minimum value while the corresponding trajectory remains at a fixed positive distance from the set of global minimizers. Consider:

\begin{equation}
    f(x,y)=e^{-x}q(y),
    \qquad
    q(y)=y^2e^{-2(y-1)},
\end{equation}
defined on $\mathbb R^2$, we will show that $f(x(t),y(t)) \to f^{\ast}$ however the trajectories are at a fixed distance away from the global minimizer set in the limit. Since $q(y)\ge 0$ and $q(y)=0$ if and only if $y=0$, we have 
\begin{equation}
    f^\ast=0,
    \qquad
    S:=\arg \min f
    =
    \{(x,0):x\in\mathbb R\}.
\end{equation}
The gradient is
\begin{equation}
    \nabla f(x,y)
    =
    \begin{bmatrix}
        -e^{-x}q(y)\\[2mm]
        e^{-x}q'(y)
    \end{bmatrix}
    =
    \begin{bmatrix}
        -f(x,y)\\[2mm]
        e^{-x}q'(y)
    \end{bmatrix}.
\end{equation}
Therefore
\begin{equation}
    \|\nabla f(x,y)\|
    \ge
    |f(x,y)|.
\end{equation}
Since $f^\ast=0$ and $f(x,y)\ge0$, this implies
\begin{equation}
    \|\nabla f(x,y)\|
    \ge
    f(x,y)-f^\ast.
\end{equation}
Thus $f$ satisfies a \kinfPLI\ with $\alpha(r)=r.$ Now consider the gradient flow of $f$ given by:
\begin{equation}
    \dot x=e^{-x}q(y),
    \qquad
    \dot y=-e^{-x}q'(y).
\end{equation}
Since
\begin{equation}
    q'(y)=2y(1-y)e^{-2(y-1)},
\end{equation}
we have $q'(1)=0$ and $q(1)=1$. Hence the line $y=1$ is invariant.
Starting from an initial condition $(x_0,1)$, the dynamics reduce to
\begin{equation}
    \dot y=0,
    \qquad
    \dot x=e^{-x}.
\end{equation}
Thus
\begin{equation}
    y(t)=1,
    \qquad
    x(t)=\log(t + e^{x_0})
\end{equation}
so
\begin{equation}
    x(t)\to+\infty.
\end{equation}

Along this trajectory,
\begin{equation}
    f(x(t),1)
    =
    e^{-x(t)}
    =
    \frac{1}{t+e^{x_0}}
    \to 0.
\end{equation}
Therefore the Lyapunov function $V(x,y)=f(x,y)-f^\ast=f(x,y)$ converges to zero. However,
\begin{equation}
    \operatorname{dist}\big((x(t),y(t)),S\big)
    =
    \operatorname{dist}\big((x(t),1),S\big)
    =
    1
\end{equation}
for all $t\ge0$. Hence the trajectory does not converge to the minimizer set in state, even though $f$ satisfies a global \kinfPLI. This shows that a \kinfPLI\ implies convergence of the objective value, but it does not imply
\begin{equation}
    \operatorname{dist}(x(t),S) \to 0.
\end{equation}

\begin{figure}[t]
    \centering
    \includegraphics[width=0.8\columnwidth]{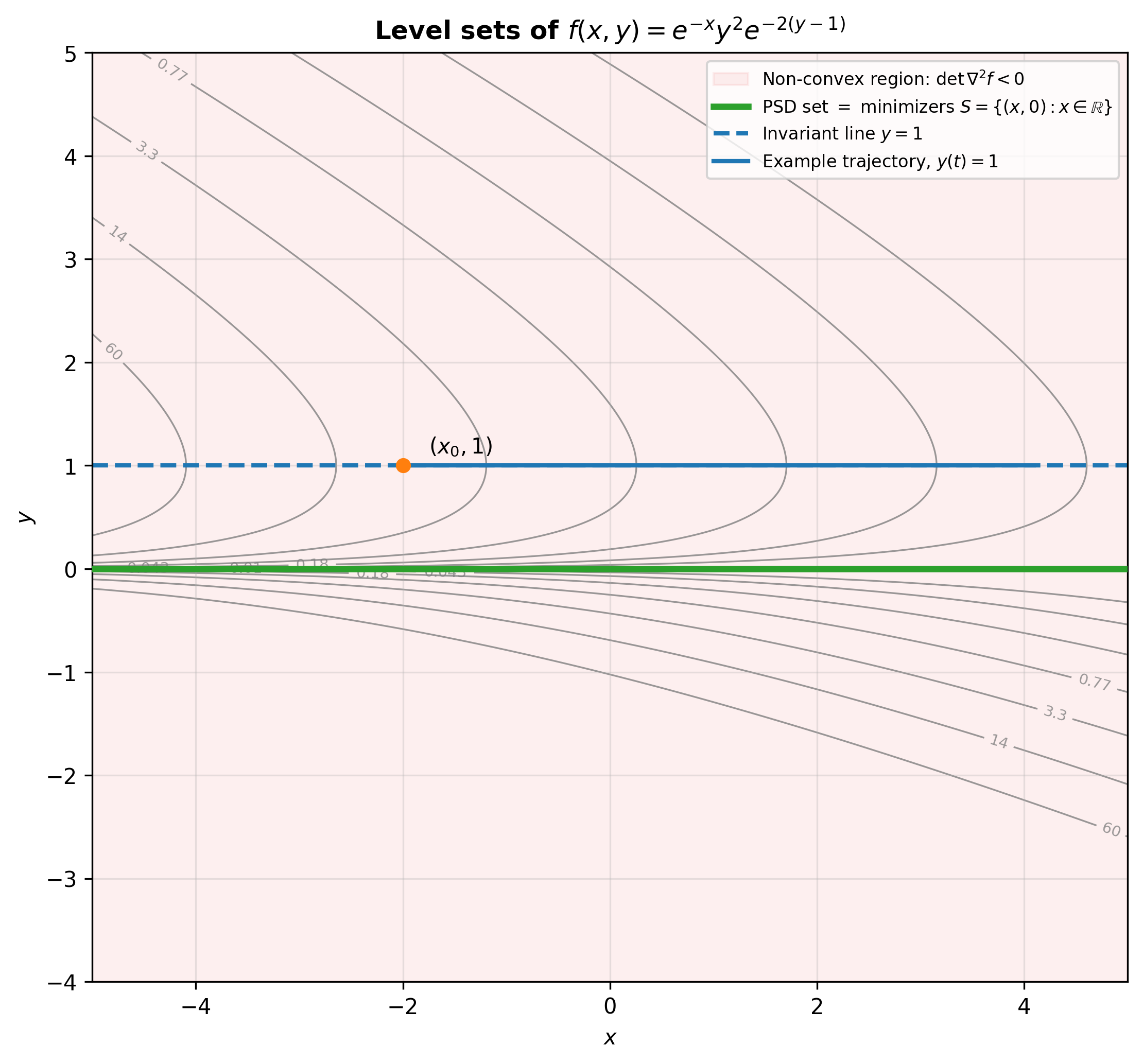}
    \caption{Schematic illustration of a trajectory of the gradient flow of the \kinfPLI\ function described in the beginning of the section. The level sets of the function are the gray lines. The blue line illustrates an invariant of the gradient flow with the resulting trajectory going towards $\infty$ while being equidistant to the set of minimizers $S$.}
    \label{fig:gf-escape}
\end{figure}

In \cite{Wensing_2020} it was shown that if the system is weakly contractive under a Riemannian metric then all of the pathologies involving possible non-convergence of gradient flow are avoided. Namely, for any given $x$ there always exists $x^{\ast} \in \arg\min f$ such that: 

\begin{equation}
    \|x(t) - x^{\ast}\| \to 0.
\end{equation}

Our main goal in this section is to generalize the above to incrementally stable systems under more general metrics and to show that, under an additional Lyapunov stability condition on each equilibrium point, every ``well behaved'' gradient system is also incrementally stable.

\subsection{Incrementally Stable Systems}

\begin{assumption}
    Throughout this subsection, unless otherwise specified, incremental stability is understood with respect to a proper metric $d$ on $\mathbb{R}^n$ inducing the Euclidean topology, restricted to $G$, and set distances $\dist$ are measured with the same metric $d$.
\end{assumption}

There has been progress in \cite{Duvall,jafarpour2021weak} in characterizing the qualitative behavior of $\omega$-limit sets when the system is non-expansive. Namely, \cite{Duvall} showed that if a system is non-expansive with respect to a strictly convex norm and has at least one bounded solution, then the $\omega$-limit set of every trajectory is also an $\omega$-limit set of a linear time-invariant system. \cite{jafarpour2021weak} showed that either trajectories of non-expansive systems in $\ell_{1}$ or $\ell_{\infty}$ are bounded and converge to a point or are unbounded. Another structural result for more general metrics came for the restricted setting of gradient flows in \cite{Wensing_2020} which showed that any weakly-contractive gradient flow possesses a path-connected set of equilibrium points. 

A different behavior of IS systems was established in \cite{centorrino2024weakly}. There, it was shown that for globally non-expansive dynamics, the existence of a unique locally exponentially stable equilibrium upgrades this local property to a global convergence result. Specifically, trajectories converge toward the equilibrium at least linearly outside a ball of radius R, while convergence is exponential inside the ball. This provides a local-to-global stability result for non-expansive dynamics.

In the following we generalize some of the results above to IS systems by:

\begin{enumerate}
    \item proving topological properties of $\omega$-limit sets;
    \item proving a local-to-global result for asymptotically stable sets;
    \item and under a local asymptotic stability assumption, proving the equivalence of IS and global semistability.
\end{enumerate}

\begin{lemma}\label{lemma:local-to-global} 
Assume $G$ is connected and forward invariant under \eqref{eq:ode}. Suppose \eqref{eq:ode} is IS on $G$, and suppose $S\subseteq G$ is locally asymptotically stable relative to $G$. Then $B_S=G$. Hence $S$ is globally asymptotically stable relative to $G$. 
\end{lemma}

The proof is provided in the appendix. The basic idea is to prove the above via contradiction. Indeed, suppose $B_S \neq G$, then evaluate points at a distance $\epsilon$ away from each other, such that one lies inside $B_S$ and the other on the boundary of $B_S$. Since one lies inside $B_S$, it must converge to $S$ while the other must stay outside $B_S$. This we show is a violation of the incremental stability condition. 


To prove that $S$ is (path-)connected we assume that $S$ consists of equilibria of the system. First we show that the endpoint map is well defined and continuous:

\begin{lemma}\label{lem:omega-limit-singleton}
Assume $G$ is forward invariant under \eqref{eq:ode} and \eqref{eq:ode} is IS on $G$. Let $S \subseteq G$ be a nonempty set of equilibria such that, for every $y \in G$, the $\omega$-limit set $\omega(y)$ contains a point of $S$; this holds if $G$ is closed and connected and $S$ is closed relative to $G$ and a locally asymptotically stable set of equilibria relative to $G$ (cf. Lemma \ref{lemma:local-to-global}). Then the endpoint map $P : G \to S$ is well defined and satisfies $d(P(x),P(y)) \le \sigma(d(x,y))$ for all $x,y \in G$, so $P$ is uniformly continuous with respect to $d$.
\end{lemma}

The proof is provided in the appendix. Combining the previous results we reach the following equivalence between IS systems and global semistability:

\begin{theorem}\label{thm:equivalence-limit-map-I.S}
Let $G$ be a closed (path-)connected set that is forward invariant under \eqref{eq:ode}. Let $S$ denote the equilibrium set of \eqref{eq:ode} and assume it is nonempty. Then the following two statements are equivalent:

\begin{enumerate}
    \item There exists a proper compatible metric $d$ such that the system is IS with respect to $d$ and the equilibrium set $S$ is locally asymptotically stable relative to $G$, with set distances measured using $d$.
    \item For every $x\in G$, the endpoint map $P(x)$ exists and belongs to $S$, and each equilibrium of $S$ is Lyapunov stable relative to $G$, that is, the system is globally semistable relative to $G$.
\end{enumerate}

Furthermore the endpoint map $P$ is continuous, and $S$ is (path-)connected.
\end{theorem}

The proof of the above is provided in the appendix. Without local asymptotic stability of the set $S$ this appears to be a harder question. In \cite{skedlak2007structure} it was shown that if $\sigma(r) = lr$, for some finite constant $l$ which is close enough to $1$, and $d(x,y) = \|x-y\|$, then one can show that the set of equilibrium points is non-empty and a retract of $G$. For arbitrary $l$, however, which is a restricted class of IS systems, the question appears to be open.

We can also generalize Proposition 3 of \cite{Wensing_2020} to show that, for an IS system admitting a Lyapunov function, LaSalle-type convergence to the set of equilibria is upgraded to convergence of each trajectory to a single equilibrium, with a continuous endpoint map. This upgrades convergence in the objective value to convergence in the state.

\begin{lemma}\label{lemma:la-Salle-I.S}
    Suppose \eqref{eq:ode} is IS on a forward-invariant (path-)connected set $G$ and one bounded trajectory exists. Assume further that either $G$ is closed, or every positive orbit starting in $G$ has compact closure contained in $G$. Further assume for \eqref{eq:ode} there exists a Lyapunov function $V(x)$, continuously differentiable on a neighborhood of $G$, satisfying:
    \begin{equation}
        \dot{V} = \frac{\partial V(x)}{\partial x}g(x) \le 0, \;\; \forall \; x \in G
    \end{equation}
    Then all trajectories converge to the largest invariant set contained in $\{x \in G: \dot{V}(x) = 0\}$. Furthermore, if this largest invariant set is the set of equilibria of \eqref{eq:ode}, then the endpoint map $P$ exists and is continuous.
\end{lemma}

The lemma above highlights a significant distinction between LaSalle-type arguments, which can be used to establish asymptotic convergence to an invariant set, and the convergence behavior imposed by incremental stability. In particular, incrementally stable systems are substantially more restrictive: when convergence to equilibrium occurs, the endpoint map $P$ must exist (and in fact by Lemma \ref{lem:omega-limit-singleton} be continuous), hence the trajectories must converge to a single equilibrium point and not orbit around the set of equilibria. 

\subsection{Application to Generalized \PLI\ systems}

The following lemma is a classical result:

\begin{lemma}\label{lemma:PL-global-semistability}{\cite{rebjock2025fast}}
    Suppose $\loss$ satisfies a \gPLI. Then all trajectories converge to a limit and the equilibria are Lyapunov stable.
\end{lemma}

Now we provide an alternative way to prove that the equilibrium set is path-connected with the machinery we constructed:

\begin{corollary}\label{cor:PLI}
    If $\loss$ satisfies a \gPLI\ globally then the gradient flow is IS with respect to the proper compatible metric constructed in the proof of Theorem \ref{thm:equivalence-limit-map-I.S} and the endpoint map is continuous, hence $S$ is path-connected.
\end{corollary}

\begin{proof}
    By Lemma \ref{lemma:PL-global-semistability} and Theorem \ref{thm:equivalence-limit-map-I.S} the system is IS, which implies the endpoint map is continuous and therefore the equilibrium set is path-connected.
\end{proof}

We remark that, under the stronger smoothness assumptions of \cite{boumal2026smooth}, a stronger regularity result is available: the set of global minima is a smooth manifold and the endpoint map is a smooth submersion onto the manifold; we do not claim this conclusion under the $C^2$ hypotheses used here. On the other hand, the incremental stability tools allow us to show the class of functions satisfying a \gPLI\ is only a sufficient criterion for the condition (2) in Theorem \ref{thm:equivalence-limit-map-I.S} to be satisfied. Indeed, one can consider the following class of objectives, which is neither weaker nor stronger than the \gPLI\ class.

\begin{assumption}\label{assump:4}
The function $\loss$ admits a $K$-reachable sub-level set (for a definition of $K$-reachability see \cite{contractive-with-inputs-sontag}), for some $\epsilon > 0$ and $K \ge 1$:
\begin{equation}
    \mathcal{T}_{\epsilon} = \{ x \quad | \quad f(x) \le f^* + \epsilon \}
\end{equation}
such that $\nabla^2 \loss(x) \succeq 0$ for $x \in \mathcal{T}_{\epsilon}$ (i.e., $f$ is locally convex on $\mathcal{T}_{\epsilon}$). Globally $f$ satisfies a \kcPLI.
\end{assumption}

\begin{proposition}\label{prop:continuity-of-inf}
Suppose $f$ satisfies Assumption \ref{assump:4}. Then the gradient flow is globally semistable, its endpoint map $P$ is continuous on $\mathbb{R}^n$ and hence $S$ is path-connected. Furthermore, the system is IS with respect to the proper compatible metric constructed in the proof of Theorem \ref{thm:equivalence-limit-map-I.S}.
\end{proposition}

\begin{proof}
Firstly we claim that given an initial condition $x \in \mathbb{R}^n$ there exists a finite time $T$ such that $\varphi(x,T)\in\mathcal{T}_{\epsilon}$. Indeed, suppose by
contradiction that
\begin{equation}
f(\varphi(x,t)) \geq f^*+\epsilon,
\qquad \forall t\geq 0.
\end{equation}
Since $f$ is non-increasing along the gradient flow,
\begin{equation}
\epsilon
\leq f(\varphi(x,t))-f^*
\leq f(x)-f^*,
\qquad \forall t\geq0.
\end{equation}
By the \kcPLI\ assumption,
\begin{equation}
\tau=
\inf_{s\in[\epsilon,f(x)-f^*]}\alpha(s)>0.
\end{equation}
Therefore,
\begin{equation}
\begin{aligned}
f(\varphi(x,T))
&=
f(x)
-
\int_0^T
\|\nabla f(\varphi(x,t))\|^2\,dt \\
&\leq
f(x)-\tau^2T.
\end{aligned}
\end{equation}
For
\begin{equation}
T > \frac{f(x)-f^*-\epsilon}{\tau^2},
\end{equation}
we obtain $f(\varphi(x,T))<f^*+\epsilon$, which is a contradiction. Hence, for every initial condition $x\in\mathbb{R}^n$, there exists a finite time $T$ such that $\varphi(x,T)\in\mathcal{T}_{\epsilon}$. Moreover, since $f$ is non-increasing along the gradient flow, $\mathcal{T}_{\epsilon}$ is forward invariant.

Since $\mathcal{T}_{\epsilon}$ is $K$-reachable and
$\nabla^2 f(x)\succeq0$ for all
$x\in\mathcal{T}_{\epsilon}$, standard logarithmic norm theory
(see Theorem 1 and Lemma 1 in
\cite{contractive-with-inputs-sontag}) gives
\begin{equation}\label{eq:K-uniform-lipschitzian}
    \|\varphi(x,t)-\varphi(y,t)\|
    \leq
    K\|x-y\|,
    \quad
    \forall x,y\in\mathcal{T}_{\epsilon},\;\forall t\geq0.
\end{equation}
Hence the system is IS on $\mathcal{T}_{\epsilon}$. We can therefore apply Lemma \ref{lemma:la-Salle-I.S} to the system restricted to the closed set $\mathcal{T}_{\epsilon}$ with $V = f$, to conclude that every trajectory inside $\mathcal{T}_{\epsilon}$ converges to an equilibrium of $f$. Now let $x\in \mathbb{R}^n$. From the first part of the proof, there exists a finite $T$ such that
\begin{equation}
x_{\epsilon}=\varphi(x,T)\in\mathcal{T}_{\epsilon}.
\end{equation}
By Lemma \ref{lemma:la-Salle-I.S},
\begin{equation}
\lim_{t\to\infty}\varphi(x_{\epsilon},t)\in S.
\end{equation}
Using the semigroup property,
\begin{equation}
\varphi(x,T+t)
=
\varphi(\varphi(x,T),t)
=
\varphi(x_{\epsilon},t),
\end{equation}
and hence $P(x)= \lim_{t\to\infty}\varphi(x,t)$ exists and belongs to $S$ for any $x \in \mathbb{R}^n$. It remains to show that every equilibrium in $S$ is Lyapunov stable relative to $\mathbb{R}^n$. Let $x^*\in S$. By the \kcPLI, $f(x^{\ast}) = f^{\ast} < f^{\ast}  + \epsilon$; by continuity of $f$ this implies $x^{\ast}$ is an interior point of $\mathcal{T}_{\epsilon}$ and there exists $r>0$ such that $B_r(x^*)\subset\mathcal{T}_{\epsilon}.$ Given $\delta>0$, choose  $\eta < \min\left\{r,\frac{\delta}{K}\right\}$. Then, for any $x$ such that $\|x-x^*\|<\eta$, both $x,x^*\in\mathcal{T}_{\epsilon}$ and, since $x^*$ is an equilibrium, \eqref{eq:K-uniform-lipschitzian} gives
\begin{equation}
\|\varphi(x,t)-x^*\|
=
\|\varphi(x,t)-\varphi(x^*,t)\|
\leq
K\|x-x^*\|
<
\delta,
\end{equation}
$\forall t \geq 0$. Thus every equilibrium in $S$ is Lyapunov stable relative to $\mathbb{R}^n$. We have therefore verified condition $2$ of Theorem \ref{thm:equivalence-limit-map-I.S}. It follows that the system is IS on $\mathbb{R}^n$ with respect to the metric constructed in its proof. Furthermore, Theorem \ref{thm:equivalence-limit-map-I.S} implies that the endpoint map $P$ is continuous and that $S$ is path-connected.
\end{proof}

\subsection{Incremental Stability of Other Optimization Algorithms}

We consider two inertial frameworks for the unconstrained minimization of an objective function: the Heavy-Ball (HB) method and the Dynamical Inertial Newton (DIN) system. Our primary concern is qualitative, namely, establishing convergence of the associated continuous-time trajectories to stationary points under suitable regularity assumptions on $f$. We note that this question has already been largely tackled and we simply provide a perspective from the point of view of incremental stability which eases the analysis in some parts and, through Theorem 5, provides an alternative route to the continuity of the endpoint map.  For convergence-rate analyses of HB and its discrete-time counterpart, we refer the reader to \cite{polyak,ghadimi2015global,lessard2016analysis}. Convergence rates for continuous-time Hessian-driven inertial dynamics and the analysis of related discrete-time, first-order implementations are developed in \cite{attouch2016fast} \cite{attouch2022first}.

The HB method was introduced in \cite{polyak}; its continuous counterpart is the damped inertial dynamics
\begin{equation}\label{eq:heavy-ball-second-order}
    \ddot x + \gamma \dot x + \nabla f(x) = 0, \quad \gamma > 0 .
\end{equation}
Introducing the velocity $v = \dot x$, \eqref{eq:heavy-ball-second-order} lifts to the first-order system
\begin{equation}\label{eq:heavy-ball}
\begin{aligned}
    \dot x &= v, \\
    \dot v &= -\gamma v - \nabla f(x).
\end{aligned}
\end{equation}
On the other hand, DIN, introduced in \cite{alvarez2002second}, adds a Hessian damping term to HB:
\begin{equation}\label{eq:din-second-order}
    \ddot x + \gamma \dot x + \beta \nabla^2 f(x)\, \dot x + \nabla f(x) = 0,
    \quad \gamma > 0,\ \beta > 0.
\end{equation} 
Although \eqref{eq:din-second-order} involves second-order derivatives of $f$, it admits an equivalent first-order representation that requires gradients only. We focus on the setting where $\gamma,\beta >0$, $\gamma\beta > 1$ as it was shown in Section 6.3 of \cite{alvarez2002second} that the system then admits a first order representation given by a gradient system.  Namely the energy function is determined for $a,b>0$ as:
\begin{equation}
    \mathcal{E}(x,z) = \beta f(x) + \frac{1}{2}\|\sqrt{a}x + \sqrt{b}z\|^2
\end{equation}
and the resulting gradient system that corresponds to \eqref{eq:din-second-order} is:
\begin{equation}\label{eq:dynamic-inertial-system}
\begin{aligned}
    \dot z &= -\sqrt{ab} x - bz , \\
    \dot x &= -\beta\nabla f(x) - a x - \sqrt{ab} z,
\end{aligned}
\end{equation}
with $\beta = \frac{1}{b}$ and $\gamma = a + b$. In particular, $a = \gamma - \frac{1}{\beta} > 0$ if and only if $\gamma\beta > 1$ and $b>0$ if and only if $\beta > 0$.

The state of \eqref{eq:dynamic-inertial-system} is $(x,z)$, not the position-velocity pair that one gets when writing \eqref{eq:din-second-order} as a first order system. Writing $v=\dot x$, the second equation in \eqref{eq:dynamic-inertial-system} gives $z=-(v+\beta\nabla f(x)+ax)/\sqrt{ab}$, so the map
\begin{equation}\label{eq:DIN-H}
    H(x,v)=\Big(x,\,-\frac{v+\beta\nabla f(x)+ax}{\sqrt{ab}}\Big)
\end{equation}
is a homeomorphism of $\mathbb{R}^{2n}$ with continuous inverse $H^{-1}(x,z)=(x,\,-\sqrt{ab}\,z-\beta\nabla f(x)-ax)$. Section 6.3 of \cite{alvarez2002second} demonstrates that using $\beta=\frac{1}{b}$ and $\gamma=a+b$, $H$ maps solutions of the first-order form of \eqref{eq:din-second-order} onto solutions of \eqref{eq:dynamic-inertial-system}.

The study of asymptotic solutions of the two methods above has been widely tackled. The first results came from a local perspective: given an initial condition sufficiently close to the set of global minimizers, \cite{polyak} showed that the HB gradient-like flow converges exponentially to a global minimizer of $f$. Global convergence to stationary points for HB was proven under the assumption of bounded trajectories and analytic objectives in \cite{haraux}. In \cite{alvarez2002second} a similar result to \cite{haraux} but for the DIN was obtained: for any analytic objective the bounded trajectories of DIN converge to the set of stationary points of the cost function. Moreover, \cite{alvarez2002second} showed that under convex objectives the boundedness assumption may be removed, for HB with a convex objective admitting a minimizer, convergence of every trajectory to a minimizer was established in \cite{alvarez2000minimizing,attouch2000heavy}. The results in \cite{attouch2000heavy, alvarez2002second} were presented under general scalar functions in an appropriate Hilbert space $H$. Here we pursue a weaker result where $H = \mathbb{R}^n$ to prove the applicability of using the theory of incrementally stable systems. Whether the same results carry over general Hilbert spaces is the subject of future work. 

\begin{proposition}\label{prop:DIN}
    Let $\gamma,\beta > 0$ with $\gamma\beta>1$, and let $f$ be a $C^2$ objective that is either convex with a nonempty set of minimizers or satisfies a \gPLI. Then DIN, in the coordinates $(x,z)$ of \eqref{eq:dynamic-inertial-system},
    defines an IS system; hence, its endpoint map $P$ is well defined and continuous. In the convex case, IS holds in the Euclidean norm, and in the \gPLI\ case, it holds with respect to the metric constructed in the proof of Theorem \ref{thm:equivalence-limit-map-I.S}.
\end{proposition}

\begin{proof}
    As we previously stated, DIN corresponds to a gradient system whose cost function is given by:
    \begin{equation}\label{eq:equiv-gradient-system-DIN}
        \mathcal{E}(x,z) = \beta f(x) + \frac{1}{2}\|\sqrt{a}x + \sqrt{b}z\|^2.
    \end{equation}
    If $f(x)$ is convex then so is $\mathcal{E}(x,z)$. The gradient flow of any convex $C^2$ function is incrementally stable and in fact non-expansive in the $\ell_2$ norm by Remark \ref{rm:inf-contracting-convex}. Since $f$ has a minimizer $x^{\ast}$, the point $(x^{\ast}, - \sqrt{\frac{a}{b}}x^{\ast})$ is an equilibrium of \eqref{eq:dynamic-inertial-system} and hence a bounded trajectory. System \eqref{eq:dynamic-inertial-system} defines a gradient system, so the Lyapunov function $V(x,z) = \mathcal{E}(x,z)$ satisfies the conditions of Lemma \ref{lemma:la-Salle-I.S} with the largest invariant set $\{\dot{V} = 0\}$ being the stationary points of $\mathcal{E}(x,z)$. Hence the endpoint map $P$ of \eqref{eq:dynamic-inertial-system} exists and is a continuous function. Now suppose $f$ satisfies a \gPLI; we seek to show that $\eqref{eq:equiv-gradient-system-DIN}$ satisfies a \gPLI. To that end notice that:
    \begin{equation}
    \begin{aligned}
        & \|\nabla \mathcal{E}(x,z)\|^2 = \\
        & \begin{bmatrix}
        \nabla f(x) \\
        \sqrt{a}x + \sqrt{b}z
        \end{bmatrix}^{\top}
        \begin{bmatrix}
        \beta^2 & \beta\sqrt{a} \\
        \beta\sqrt{a} & (a+b)
        \end{bmatrix}
        \otimes I
        \begin{bmatrix}
        \nabla f(x) \\
        \sqrt{a}x + \sqrt{b}z.
        \end{bmatrix}
    \end{aligned}
    \end{equation}
    The global minimizers of $\mathcal{E}(x,z)$ have value $\beta f^{\ast}$. Assume $f$ satisfies a \gPLI\ with constant $\xi_f > 0$. Then, for $\xi_{\mathcal{E}} > 0$,
    \begin{equation}
    \begin{aligned}
    & \xi_{\mathcal{E}}(\mathcal{E}(x,z)  - \mathcal{E}^{\ast}) = \xi_{\mathcal{E}}(\beta(f(x) - f^{\ast}) + \frac{1}{2}\|\sqrt{a}x + \sqrt{b}z\|^2) \\
    & \le \xi_{\mathcal{E}}(\frac{\beta}{\xi_f}\|\nabla f\|^2 + \frac{1}{2}\|\sqrt{a}x + \sqrt{b}z\|^2),
    \end{aligned}
    \end{equation}
    which can be written as the quadratic form:
    \begin{equation}
    \begin{bmatrix}
    \nabla f(x) \\
    \sqrt{a}x + \sqrt{b}z
    \end{bmatrix}^{\top}
    \begin{bmatrix}
    \frac{\xi_{\mathcal{E}}\beta}{\xi_f} & 0 \\
    0 & \frac{\xi_{\mathcal{E}}}{2}
    \end{bmatrix}
    \otimes I
    \begin{bmatrix}
    \nabla f(x) \\
    \sqrt{a}x + \sqrt{b}z
    \end{bmatrix}.
    \end{equation}
    So to prove that $\mathcal{E}(x,z)$ satisfies a \gPLI, it suffices to choose $\xi_{\mathcal{E}}$ such that:
    \begin{equation}
    \begin{bmatrix}
        \beta^2  & \beta\sqrt{a} \\
        \beta\sqrt{a} & (a+b)
    \end{bmatrix}
    - 
   \xi_{\mathcal{E}} \begin{bmatrix}
    \frac{\beta}{\xi_f} & 0 \\
    0 & \frac{1}{2}
    \end{bmatrix}
    \succeq 0,
    \end{equation}
    which is always possible since:
    \begin{equation}
    \det(
    \begin{bmatrix}
        \beta^2 & \beta\sqrt{a} \\
        \beta\sqrt{a} & (a+b)
    \end{bmatrix}) = \beta^2b> 0
    \end{equation}
    and $\beta^2 > 0$, so the above matrix is clearly positive-definite. By Corollary \ref{cor:PLI} the system is IS and $P$ is a well defined continuous function.
\end{proof}

\begin{remark}\label{rm:DIN-original-coordinates}
Although Proposition~\ref{prop:DIN} is stated in the coordinates $(x,z)$, the homeomorphism $H$ in \eqref{eq:DIN-H} induces a metric in the original coordinates $(x,v)$ by pushing points forward to the $(x,z)$ space and measuring distances there: if $d$ denotes the metric of Proposition~\ref{prop:DIN}, set $d_H(u,w)=d(H(u),H(w))$. Since $H$ conjugates the two flows, DIN in the coordinates $(x,v)$ is IS with respect to $d_H$. Indeed, letting $\varphi_{xz}$ denote the flow of \eqref{eq:dynamic-inertial-system} and $\varphi_{xv}$ the flow of the first order system form of \eqref{eq:din-second-order}:
\begin{equation}
\begin{aligned}
    & d_H(\varphi_{xv}(u,t),\varphi_{xv}(w,t))=d(H(\varphi_{xv}(u,t)),H(\varphi_{xv}(w,t))) \\
    & = d(\varphi_{xz}(H(u),t),\varphi_{xz}(H(w),t)) \le \sigma(d(H(u),H(w))) \\
    & = \sigma(d_H(u,w)),
\end{aligned}
\end{equation}
and its endpoint map $H^{-1}\circ P\circ H$ is well defined and continuous. As $H$ is a homeomorphism of $\mathbb{R}^{2n}$, $d_H$ is again proper and compatible. In the convex case this gives IS with respect to $\|H(u)-H(w)\|$, not with respect to the Euclidean norm in $(x,v)$.
\end{remark}

Interestingly since the HB method is not directly the gradient descent of some energy function, proving that the system is incrementally stable for convex functions via a Euclidean or Riemannian metric appears to be harder. We therefore leverage the results obtained in \cite{attouch2000heavy, alvarez2000minimizing} to conclude incremental stability via Theorem \ref{thm:equivalence-limit-map-I.S}.

\begin{proposition}
    For any $\gamma > 0$ and any $C^2$ convex objective with a non-empty set of minimizers, HB defines an IS system (with respect to the metric constructed in the proof of Theorem \ref{thm:equivalence-limit-map-I.S}); hence, its endpoint map $P$ is well defined and continuous on $\mathbb{R}^{2n}$.
\end{proposition}

\begin{proof}
    It was proven in \cite{attouch2000heavy, alvarez2000minimizing} that for a convex objective with a minimizer, the endpoint map $P$ of \eqref{eq:heavy-ball} is well defined. Therefore, it remains to prove that every equilibrium is Lyapunov stable. Consider the Lyapunov function defined with respect to some $x^{\ast} \in \arg\min f$:
    \begin{equation}
        V(x,v) = \frac{1}{2}\|v\|^2 + \frac{\gamma\epsilon}{2}\|x - x^{\ast}\|^2  + \epsilon\langle x-x^{\ast},v \rangle + f(x) - f^{\ast},
    \end{equation}
    where $\epsilon \in (0,\gamma)$. Its derivative with respect to the HB flow satisfies:
    \begin{equation}
    \begin{aligned}
        \dot{V}(x,v) \le -(\gamma-\epsilon)\|v\|^2 - \epsilon \langle x - x^{\ast}, \nabla f(x) \rangle \le 0,
    \end{aligned}
    \end{equation}
    where the last inequality is obtained via the convexity of $f$ and given that $\nabla f(x^{\ast}) = 0$.  $V(x,v)$ has a unique global minimizer at $(x^{\ast},0)$ and is coercive. Therefore, given any ball $B_{\delta}((x^\ast, 0))$, there exists $\delta' > 0$ such that the sublevel set $\{V \le \delta'\}$ is contained in $B_{\delta}((x^{\ast},0))$; by the inequality above, this sublevel set is forward invariant, and it contains a neighborhood of $(x^{\ast},0)$. Hence the point $(x^{\ast},0)$ is Lyapunov stable and since $x^{\ast}$ was arbitrary every equilibrium of \eqref{eq:heavy-ball} is Lyapunov stable.  Applying Theorem \ref{thm:equivalence-limit-map-I.S} one concludes that \eqref{eq:heavy-ball} is an incrementally stable system and its endpoint map $P$ is continuous.
    
\end{proof}

\paragraph*{A note on sum-of-squares (SOS) programs proving incremental stability}
Building on the convex framework developed in \cite{Aylward_2008}, one could search for contraction metrics $M(\xi)$, where $\xi$ denotes the state of \eqref{eq:heavy-ball} or \eqref{eq:dynamic-inertial-system}, represented as SOS matrix polynomials satisfying $M(\xi) \succeq \underline{\lambda} I$ for some $\underline{\lambda}>0$, and certifying weak contractivity of either \eqref{eq:heavy-ball} or \eqref{eq:dynamic-inertial-system}. Such an approach would require the objective function to be polynomial and, consequently, analytic. To establish the relevance of this convex search procedure, it is natural to focus on non-coercive objectives. Indeed, the results in \cite{haraux,alvarez2002second} establish pointwise convergence, from every initial condition, to a single point in the equilibrium set of each system, when the objective is coercive and analytic. In the non-coercive case, by contrast, pointwise convergence is established only under the additional assumption that the trajectory is bounded. Therefore, if one can construct a metric certifying weak contractivity of either second-order method, then Proposition \ref{prop:equivalence-of-topologies} implies that the distance induced by the metric is proper and generates the same topology as the Euclidean metric on $\mathbb{R}^{2n}$. Consequently, provided $f$ has a critical point (so that a bounded trajectory exists), Lemma \ref{lemma:la-Salle-I.S} can be invoked to establish continuity of the endpoint map.

\section{Conclusion}

We studied incremental stability of gradient flows satisfying generalized Polyak--\L{}ojasiewicz inequalities. For coercive objectives with a unique, locally strongly convex minimizer and a \kcPLI, we constructed an explicit Riemannian contraction metric whose distance is proper and compatible. Furthermore, under the curvature assumptions of Theorems~\ref{thm:concavity-on-annulus} and~\ref{prop:above-PL-regime} the Riemannian contraction metric is uniformly equivalent to the Euclidean metric. For multiple minimizers, we showed that IS with respect to a proper compatible metric, together with local asymptotic stability of the equilibrium set, is equivalent to global semistability, which implies continuity of the endpoint map and (path-)connectedness of the equilibrium set. We applied these results to classical \PLI\ objectives, to a new class of locally convex objectives, and to the Heavy-Ball and Dynamical Inertial Newton flows. Extensions to Hilbert spaces and sum-of-squares certificates are left for future work.

\bibliographystyle{IEEEtran}
\bibliography{ref.bib}

\section{Appendix}

\subsection{Proof of Theorem \ref{prop:global}}

To prove Theorem \ref{prop:global}, we first state and prove a Lemma used to interpolate between a constant and a state-dependent metric.

\begin{lemma}\label{lemma:partition-of-metrics}
Suppose $\loss$ satisfies Assumption \ref{assump:1}. Define a continuous scalar function $v : \mathbb{R} \to \mathbb{R}$ which is non-decreasing and continuously differentiable on the set $v^{-1}(a-\Delta, b + \Delta)$ for some $a < b\le+\infty$ and $\Delta > 0$ (with $b+\Delta:=+\infty$ when $b=+\infty$). Define $A = (v \circ f)^{-1}(a,b)$ and assume that for $M(x) = e^{2v(f(x))}I$ (which is in $C^1((v \circ f)^{-1}(a-\Delta, b + \Delta), \mathbb{S}^n))$ we have, for the gradient flow \eqref{eq:gradient-flow}, that it satisfies
\begin{equation}
    \mathcal{L}_{M}(x) \;\le\; -\nu_1 \le 0, \; \forall x \in A,
\end{equation}
Let $A^{c}$ denote the complement of $A$. Assume further that, for all $x \in A^{c}$, one has that $\lambda_{\min}(\nabla^2 f(x)) \ge \nu \ge 0$. Then there exists a metric $\hat{M}(x) \in C^1(\mathbb{R}^n,\mathbb{S}^n)$ which satisfies:
\begin{equation}
    \mathcal{L}_{\hat{M}}(x) \;\le\; -\min(\nu_1,\nu) \le 0, \; \forall x \in \mathbb{R}^n
\end{equation}
\end{lemma}

\begin{proof}
    Assume first that $b<\infty$.
    Consider a smooth non-negative function $\chi : \mathbb{R} \rightarrow [0,1]$ satisfying,
    \begin{equation}
    \begin{aligned}
        & \chi(t) = 0 \quad \text{for } t \le 0 \\
        & \chi(t) = 1  \quad \text{for } t \ge 1,
    \end{aligned}
    \end{equation}
    and define,
    \begin{equation}
        h'(s) = \begin{cases}
            0 & s \le  a- \Delta\\
            \chi\left(\frac{s - (a - \Delta)}{\Delta}\right) & a- \Delta < s < a \\
            1 & a \le s \le b \\
            1 - \chi\left(\frac{s - b}{\Delta} \right) & s > b\\
        \end{cases} \\
    \end{equation}
    and set $h(s):=a+\int_a^s h'(u)\,du$, so that $h(s)$ is constant for $s \le a-\Delta$ and $s \ge b + \Delta$ and satisfies $h(s) = s$ for $s \in [a,b]$. Let $\hat{M}(x) = e^{2 h(v(f(x)))}I$. We need to show that $\mathcal{L}_{\hat{M}}(x) \le -\min(\nu_1,\nu)$ for all $x \in \mathbb{R}^n$.  First notice that  $f$ is $\nu-$strongly convex on $(v \circ f)^{-1}[v(\mloss),a-\Delta] \subset A^c$ and $h(v(f(x)))$ is constant in that set so $\hat{M} = e^{2c}I$ for some constant $c \in \mathbb{R}$ and $\mathcal{L}_{\hat{M}}(x) = \mathcal{L}_{I}(x) \le -\nu$. For $x \in (v \circ f)^{-1}(a-\Delta, a]$ we apply Lemma \ref{lemma:tightness-sysmmetric} and \ref{lemma:riemmanina-metric-with-lyapunov} to get:
    \begin{equation}\label{eq:contraction-with-classK-weighting}
    \begin{aligned}    
        & \mathcal{L}_{\hat{M}}(x) = -\lambda_{\min}(\nabla^2f(x)) - \nabla h(v(f(x)))^{\top} \nabla f(x) \\
        & = -\lambda_{\min}(\nabla^2f(x)) - h'(v(f(x)))v'(f(x)) \|\nabla f(x)\|^2 \\
        & \leq -\lambda_{\min}(\nabla^2f(x)) - h'(v(f(x)))v'(f(x))\alpha(f(x) - \mloss)^2.
    \end{aligned}
    \end{equation}
    For all $x$, $v$ is non-decreasing so $v'(f(x)) \ge 0$ and also by definition $h'(v(f(x))) \ge 0 $. Therefore since $x \in (v \circ f)^{-1}(a-\Delta, a] \subset A^c$, it still holds that $\mathcal{L}_{\hat{M}}(x) \le -\nu$. For $x \in A$ we have that $h(v(f(x))) = v(f(x))$ hence $\hat{M} = M$ and $\mathcal{L}_{\hat{M}}(x) \le -\nu_1$. Therefore $\mathcal{L}_{\hat{M}}(x) \le -\min(\nu_1,\nu)$ for all $x \in (v \circ f)^{-1}[v(\mloss),b)$.
    
    \noindent For $x \in (v \circ f)^{-1}[b, \infty) \subset A^c$ we know that $-\lambda_{\min}(\nabla^2f(x)) \leq -\nu$. Considering first $x \in (v \circ f)^{-1}[b,b+\Delta]$, we get again by using that $h'(v(f(x))) \ge 0$, $v'(f(x)) \ge 0$ and \eqref{eq:contraction-with-classK-weighting} that $\mathcal{L}_{\hat{M}}(x) \le -\nu$. For $v(f(x)) \ge b + \Delta$ we have that $h'(v(f(x))) = 0$ so for some $c_2 \in \mathbb{R}$,  $h(v(f(x))) = c_2$. Therefore we have that $\hat{M}(x) = e^{2c_2}I$ and $\mathcal{L}_{\hat{M}} \le -\nu$ for $x \in (v \circ f)^{-1}[b+\Delta,\infty)$. This finalizes the proof that $\mathcal{L}_{\hat{M}}(x) \le -\min(\nu_1,\nu)$ for all $x \in \mathbb{R}^n$. Since $v\circ f$ is $C^1$ on $(v\circ f)^{-1}((a-\Delta,b+\Delta))$, and $h$ is constant outside $(a-\Delta,b+\Delta)$ with $h'=0$ at the endpoints, it follows that $\hat{M}\in C^1(\mathbb{R}^n,\mathbb{S}^n)$.
    If $b=+\infty$, the proof follows similarly, without an upper transition, but we instead define
    \begin{equation}
        h'(s) = \begin{cases}
            0 &  s \le a -\Delta, \\
            \chi \left(\frac{s - (a-\Delta)}{\Delta} \right) & s \ge a -\Delta, \\
        \end{cases} 
    \end{equation}
    again with $h(s):=a+\int_a^s h'(u)\,du$, so that $h(s)=s$ for all $s\ge a$ and $h$ is constant for $s\le a-\Delta$.
\end{proof}

We are now ready to state the proof of Theorem \ref{prop:global}.

\begin{proof}
   First we prove that there exists $M(x)$ such that $\mathcal{L}_{M}(x) \le -\nu$ on $\mathbb{R}^n \setminus W$ where $W$ is a small open connected set around $x^*$. Let $v$ be a real-valued function and define:
   \begin{equation}
       M(x) = Ie^{2v(f(x))}.
   \end{equation}
   We have by Lemma \ref{lemma:tightness-sysmmetric} and Lemma \ref{lemma:riemmanina-metric-with-lyapunov}:
   \begin{equation}
     \mathcal{L}_{M}(x) = -\lambda_{\min}(\nabla^2f(x)) - \frac{\partial v}{\partial f(x)} \|\nabla f(x)\|^2
   \end{equation}
   If $\frac{\partial v}{\partial f} \ge 0$ then by the \kcPLI\ bound on $\loss$:
   \begin{equation}
       \mathcal{L}_{M}(x)  \le - \lambda_{\min}(\nabla^2f(x)) - \frac{\partial v}{\partial f(x)} \alpha(f(x)-f^*)^2.
   \end{equation}
    Therefore some algebraic manipulation shows that for $\mathcal{L}_{M}(x) \leq -\nu$, it is sufficient to construct a continuously differentiable $v$ such that:
    \begin{equation}\label{eq:bound-partial-g}
        \frac{\partial v}{\partial f(x)} \geq \max \{\max_{{\{z \in f^{-1}(f(x))\}}} \frac{\nu - \lambda_{\min}(\nabla^2 f(z))}{\alpha(f(x)-f^*)^2},0\}.
    \end{equation}
   Define:
    \begin{equation}
        m(x) = \nu - \lambda_{\min}(\nabla^2 f(x)),
    \end{equation}
    \begin{equation}\label{eq:m*}
        m^*(y) = \max\{m(x): f(x) = y\},
    \end{equation}
    Since $f\in C^2$, $\lambda_{\min}(\nabla^2 f(x))$ is continuous and so is $m(x)$. Because $f$ is coercive, $f^{-1}(\cdot)$ is a compact valued map and $m^*(y) = m(x_y)$ for some $x_y \in f^{-1}(y)$. Consider any sequence $y_n \to y$, we can assume without loss of generality that the sequence belongs to some compact set $[y_l, y_u]$, hence there exists a subsequence $x_n$ such that $x_n \to x$ and $m(x_n) = m^*(y_n)$. By continuity of $m(\cdot)$:
    \begin{equation}
    \begin{aligned}
        & \limsup_{y_k \to y} m^*(y_k)  = \lim_{n \to \infty} m^*(y_n) = \\ 
        & \lim_{n \to \infty} m(x_n) = m(x) \le m^*(y)
    \end{aligned}
    \end{equation}
    Therefore $m^*$ is upper semicontinuous and locally bounded above. We can then construct a piecewise linear function $m^u$ which satisfies $m^u(f(x)) > \max\{m^*(f(x)),0\}$. 
    \noindent Now let $W = f^{-1}[\mloss,\mloss+\delta)$ where $\delta$ is sufficiently small so that $\lambda_{\min}(\nabla^2 f(x)) > \nu - \epsilon$ for all $x \in f^{-1}([\mloss,\mloss+\delta])$ and define:
    \begin{equation}
        v(s) = \begin{cases}
            \int_{f^{\ast} + \frac{\delta}{2}}^{s}  \frac{m^u(\tau)}{\alpha(\tau - \mloss)^2} d\tau, & s \ge f^{\ast} + \frac{\delta}{2} \\
            0 & s \le f^{\ast} + \frac{\delta}{2}.
        \end{cases}
    \end{equation}
    Since $\alpha$ is continuous and positive on $(0,\infty)$, the denominator inside the integral expression above is continuous and strictly greater than zero for $\tau \ge f^{\ast} + \frac{\delta}{2}$, and we have that $v \in C{^1}((f^{\ast} + \frac{\delta}{2},\infty),\mathbb{R})$ is a non-decreasing function which satisfies \eqref{eq:bound-partial-g} in $\mathbb{R}^n \setminus W$. Therefore $M(x) \in C^{1}(\mathbb{R}^n \setminus W, \mathbb{S}^n)$ and $  \mathcal{L}_{M}(x) \;\le\; -\nu, \; \forall x \in \mathbb{R}^n \setminus W.$ This allows us to apply Lemma \ref{lemma:partition-of-metrics}, with $A = (v \circ f)^{-1}(v(f^* + \delta), \infty)$ (i.e., $a=v(f^*+\delta)$ and $b=+\infty$), any $\Delta \in (0, v(\mloss + \delta))$, and with $\nu - \epsilon$ in place of $\nu$, to construct $\hat{M}(x) \in C^{1}(\mathbb{R}^n,\mathbb{S}^n)$ such that:
    \begin{equation}
        \mathcal{L}_{\hat{M}}(x) \;\le\; -\nu +\epsilon, \; \forall x \in \mathbb{R}^n.
    \end{equation}    
    By construction, $M(x)$ is uniformly lower bounded since $m^u(\cdot) \ge 0$ and from the proof of Lemma \ref{lemma:partition-of-metrics}, so is $\hat{M}(x)$. So from the lower-bound estimate in the proof of Proposition \ref{prop:ies} we can also derive that there exists $c > 0$ such that:
    \begin{equation}
        \|\varphi(x_0,t) - x^{\ast}\| \le c d_{\hat{M}}(x_0,x^{\ast})e^{-(\nu - \epsilon) t}.
    \end{equation}
\end{proof}

\subsection{Proof of Theorem \ref{thm:concavity-on-annulus}}

\begin{remark}\label{rm:restricted-form-M}
    We won't prove it here in the interest of space, but for any $f$ satisfying Assumption \ref{assump:1}, $\epsilon \in (0,\nu)$ and compact set $K$ containing $x^{\ast}$, we can restrict the metric $\hat{M}$ to be of the form $e^{2cf(x)}I$ for some $c > 0$ to show that:
    \begin{equation}
        \mathcal{L}_{\hat{M}}(x) \;\le\; -\nu + \epsilon, \; \forall x \in K.
    \end{equation}
    The proof follows almost identically to Theorem 4.1 of \cite{GIESLConverse}. We will use this simpler $M$ to simplify the proof of Theorem \ref{thm:concavity-on-annulus}.
\end{remark}

\begin{proof}
    WLOG we can assume $K$ contains $x^{\ast}$. By Remark \ref{rm:restricted-form-M} there exists $c > 0$, such that for $M(x) = e^{2cf(x)}I$,  $\mathcal{L}_{M}(x) \le -l < 0$ (with $l \in (0,\nu)$, which can be made arbitrarily close to $\nu$ as $c$ increases per Theorem 4.1 of \cite{GIESLConverse}) for all $x \in K$. Since $K$ is bounded, $f$ is continuous and coercive there exists $b > \max_{x\in K}cf(x)$ such that the compact set $\mathcal{C} = (cf)^{-1}([c\mloss,b])$ satisfies $\mathcal{C} \supset K$. Strong convexity in $K^c$ shows via an application of Lemma \ref{lemma:tightness-sysmmetric} and Lemma \ref{lemma:riemmanina-metric-with-lyapunov} that for the same metric $M(x) = e^{2cf(x)}I$ we have $\mathcal{L}_{M}(x) \leq -\min\{m,l\} < 0$ for all $x \in \mathcal{C}$. Lastly $f$ is $m-$strongly convex in $(cf)^{-1}([b,\infty)) \subset K^c$ and therefore by Lemma \ref{lemma:partition-of-metrics} (with $v(s) = cs$, $a=c\mloss$, $b$ as above, any $\Delta>0$, and $A = (cf)^{-1}(cf^{\ast},b)$; since $v\circ f\ge a$, only the upper cutoff of that proof is active) there exists a metric $\hat{M}(x) \in C^{1}(\mathbb{R}^n,\mathbb{S}^n)$, which is uniformly bounded below and above, satisfying:
    \begin{equation}
        \mathcal{L}_{\hat{M}}(x) \le -\min\{m,l\} < 0, \; \forall x \in \mathbb{R}^n.
    \end{equation} 
    Hence, by Proposition \ref{prop:ies}, the gradient flow system is globally IES and $x^{\ast}$ is globally exponentially stable.
\end{proof}

\subsection{Proof of Theorem \ref{prop:above-PL-regime}}
\begin{proof}
As in Theorem \ref{prop:global}, we prove that there exists $M(x)$ such that $\mathcal{L}_{M}(x) \le -\nu$ on $\mathbb{R}^n \setminus W$ where $W$ is a small open connected set around $x^*$, chosen as in Theorem~\ref{prop:global}: $W=f^{-1}[\mloss,\mloss+\delta)$, with $\delta>0$ such that $\lambda_{\min}(\nabla^2 f(x))>\nu-\epsilon$ on $f^{-1}([\mloss,\mloss+\delta])$ for a fixed $\epsilon\in(0,\nu)$. Let $v$ be a real-valued function and define $M(x) = Ie^{2v(f(x))}$. By Lemma \ref{lemma:tightness-sysmmetric} and Lemma \ref{lemma:riemmanina-metric-with-lyapunov}, we obtain that a sufficient condition for contractivity to hold on $\mathbb{R}^n \setminus W$, given a metric $M(x) = e^{2v(f(x))}I$ and $\lambda_{\min}(\nabla^2 f(x)) \ge -m$, is:
\begin{equation}\label{eq:equivalence-contraction}
\frac{\partial v}{\partial f(x)}\|\nabla f(x)\|^2 \ge \nu + m
\end{equation}

Define, as in Theorem \ref{prop:global}, $s_0 = \mloss+\frac{\delta}{2}$,
\begin{equation}
    \psi(s) = \int_{s_0}^{s} \frac{1}{\alpha(\tau - f^*)^2} d\tau,
\end{equation}

and for some $c > 0$, let $v(s) = c \; \psi(s)$ for $s\ge s_0$ and $v(s)=0$ for $s\le s_0$, so that $v$ is continuous and non-decreasing on $\mathbb{R}$ and $C^1$ on $(s_0,\infty)$. By the Fundamental Theorem of Calculus and $\|\nabla f(x)\| \geq \alpha(f(x) - f^*)$:

\begin{equation}
\frac{\partial v}{\partial f(x)} \|\nabla f(x)\|^2  = 
\frac{c\|\nabla f(x)\|^2}{\alpha(f(x)-f^*)^2} \geq c.
\end{equation}

Choosing $c \ge \nu+m$ ensures \eqref{eq:equivalence-contraction} holds, and by \eqref{eq:strong-PLI}, $v$ is uniformly bounded above on $\mathbb{R}$ (and $v \ge 0$ everywhere), which implies $M(x)$ is uniformly bounded above and below on $\mathbb{R}^n \setminus W$. As in the proof of Theorem \ref{prop:global}, Lemma \ref{lemma:partition-of-metrics}, with $a=v(\mloss+\delta)$, $b=+\infty$, any $\Delta\in(0,v(\mloss+\delta))$, and $\nu-\epsilon$ in place of $\nu$, extends $M$ into a smooth metric $\hat{M}$ that satisfies the contractivity condition, $\mathcal{L}_{\hat{M}}(x)\le-\nu+\epsilon$, in $\mathbb{R}^n$ while keeping the property that $\hat{M}$ is uniformly bounded above and below. By Proposition \ref{prop:ies} the gradient flow system is globally IES and also $x^{\ast}$ is globally exponentially stable.
\end{proof}

\subsection{Generalization of Theorem \ref{thm:concavity-on-annulus} to Theorem \ref{thm:ies-inf-contracting-finite-time}}

The proof follows the approach developed in \cite{russo2010entrainment,contractive-with-inputs-sontag}. It could also be adapted to the setting of time-varying vector fields and $K$-reachable sets $G$ as in the above references, however for exposition purposes we take the simpler setting where the forward-invariant set is $\mathbb{R}^n$ and the vector field is autonomous.

\begin{proof}
Take any two initial conditions $x_0,z_0$ and the straight segment connecting the two denoted by $\gamma(r)=r x_0 +(1-r)z_0$ for $r \in [0,1]$. Define $\psi(t,r)=\varphi(\gamma(r),t)$ and $w(t,r)=\frac{\partial}{\partial r}\psi(t,r)$. Differentiating and using the chain rule, $\frac{\partial w}{\partial t}(t,r) = \frac{\partial}{\partial r}g(\psi(t,r)) = \nabla g(\psi(t,r)) w(t,r)$. Hence by Coppel's inequality:
\begin{equation}
    \|w(t,r)\| \le \|w(0,r)\| e^{\int_{0}^{t} \mu(\nabla g(\psi(\tau,r)))\,d\tau}.
\end{equation}
Since each trajectory \(\psi(t,r)\) spends at most time $T$ in $K^c$, and $\underline{\nu} + \overline{\nu} \ge 0$:
\begin{equation}
\int_0^t \mu\!\left(\nabla g(\psi(\tau,r))\right)\,d\tau
\le \overline{\nu}T-\underline{\nu}(t-T)
=
(\overline{\nu} + \underline{\nu})T-\underline{\nu}t.
\end{equation}
Therefore, $\|w(t,r)\|\le e^{(\overline{\nu} + \underline{\nu})T}e^{-\underline{\nu} t}\|w(0,r)\|.$ $w(0,r)=\gamma'(r)=x_0-z_0,$ and the Fundamental Theorem of Calculus gives $\varphi(x_0,t)- \varphi(z_0,t) = \int_0^1 w(t,r)\,dr,$ thus
\begin{equation}
\begin{aligned}
& \|\varphi(x_0,t)-\varphi(z_0,t)\|
\le
\int_0^1 \|w(t,r)\|\,dr \\
& \le
e^{(\overline{\nu} + \underline{\nu})T}e^{-\underline{\nu}t}\|x_0-z_0\|.
\end{aligned}
\end{equation}
\end{proof}

The proof of Corollary \ref{cor:ies-inf-contracting-compact} is now straightforward:

\begin{proof}
Let $K = \mathbb{R}^n \setminus C$, $\underline{\nu} = \nu$ and $\overline{\nu} = \max\{-\nu, \max_{x\in C} \mu(\nabla g(x))\}$, which is finite since $C$ is compact and $x \to \mu(\nabla g(x))$ is continuous. By Theorem \ref{thm:ies-inf-contracting-finite-time} it remains to be shown that for all $x \in \mathbb{R}^n$ the system spends at most finite time $T$ in $C$. Since $C$ is compact and $x^{\ast}$ is globally asymptotically stable, for any $\delta > 0$ there exists a finite time $T$ such that $\varphi(x,T) \in B_{\delta}(x^*)$ for any $x \in C$. Local stability of $x^{\ast}$ allows one to choose $\delta$ small enough such that $B_{\delta}(x^{\ast}) \cap C = \emptyset$. Lastly if for any $y \in \mathbb{R}^n$ the time $t$ map $\varphi(y,t) = x$ where $x \in C$, then we have that $\varphi(\varphi(y,t),T) \in B_{\delta}(x^{\ast})$. This concludes that trajectories starting in $C$ spend at most time $T$ in $C$ and so do trajectories entering $C$ through the flow. Theorem \ref{thm:ies-inf-contracting-finite-time} then proves the resulting system is IES.
\end{proof}

\subsection{Proofs of Incrementally Stable System}

\textit{Lemma \ref{lemma:local-to-global}}

\begin{proof}
    Assume by contradiction that $B_S \ne G$. Since $S$ is locally asymptotically stable $B_S$ is open relative to $G$, and since $G$ is connected and $B_S \ne G$, $\partial B_S \neq \emptyset$ (the relative boundary of $B_S$ with respect to $G$). Pick any $y \in \partial B_S$. Since $y \notin B_S$ (as $B_S$ is open), there exists $r > 0$ such that $\dist(\varphi(y,t), S) \ge r$ for all $t \ge 0$. Indeed one can take $r = \rho$, the radius in the definition of local asymptotic stability, since $\dist(\varphi(y,t),S) < \rho$ for some $t$ implies $\dist(\varphi(y,t),S) \to 0$ and hence $y \in B_S$. Now choose $\epsilon' \le \frac{r}{2}$ and, since $y \in \partial B_S$ and $B_S$ is open, pick $x \in B_S$ with $\sigma(d(x,y)) \le \epsilon'$. By incremental stability:
    \begin{equation}\label{eq:non-exp-prop-1}
        d(\varphi(x,t),\varphi(y,t)) \le \sigma(d(x,y)) \le \epsilon' \le \frac{r}{2}, \quad \text{for all } t \ge 0.
    \end{equation}
    Since $x \in B_S$, we have $\dist(\varphi(x,t), S) \to 0$, so there exists $N > 0$ such that for all $t > N$:
    \begin{equation}
        \dist(\varphi(x,t), S) < \frac{r}{4}.
    \end{equation}
    By the triangle inequality, for any $s \in S$:
    \begin{equation}
        d(s,\varphi(y,t)) \le d(s,\varphi(x,t)) + d(\varphi(x,t),\varphi(y,t)).
    \end{equation}
    Taking the infimum over $s \in S$:
    \begin{equation}
        \dist(\varphi(y,t), S) \le \dist(\varphi(x,t), S) + d(\varphi(x,t),\varphi(y,t)).
    \end{equation}
    Rearranging, for $t > N$:
    \begin{equation}
    \begin{aligned}
        & d(\varphi(x,t),\varphi(y,t)) \ge \dist(\varphi(y,t), S) - \dist(\varphi(x,t), S) \ge \\
        & r - \frac{r}{4} = \frac{3r}{4}.
    \end{aligned}
    \end{equation}
    This contradicts \eqref{eq:non-exp-prop-1}, and therefore $B_S = G$.
\end{proof}

\textit{Lemma \ref{lem:omega-limit-singleton}}

\begin{proof}
We first show that the $\omega$-limit set of any trajectory is compact. Fix $y \in G$ and let $y_S \in S$. Since $y_S$ is an equilibrium, $\varphi(y_S, t) = y_S$ for all $t \ge 0$. By IS:
\begin{equation}
    d(\varphi(y,t),y_S) = d(\varphi(y,t),\varphi(y_S,t)) \le \sigma(d(y,y_S)).
\end{equation}
Therefore, $\varphi(y,t)$ lies in a compact ball of $d(\cdot,\cdot)$. Since the $\omega$-limit set $\omega(y)$ is closed and the trajectory lies in a compact set, $\omega(y)$ is compact and nonempty.

By assumption, there exists $y_i \in \omega(y) \cap S$, which is an equilibrium. By definition, there exists a sequence $t_i \to \infty$ such that $\|\varphi(y,t_i)-y_i\| \to 0$, which implies, since $d$ is compatible, that $d(\varphi(y,t_i),y_i) \to 0$. Given $\epsilon > 0$, choose $t_{i'}$ large enough such that
\begin{equation}
    \sigma(d(\varphi(y, t_{i'}),y_i)) \le \epsilon.
\end{equation}
Since $y_i \in S$ is an equilibrium, $\varphi(y_i, k) = y_i$ for all $k \ge 0$. By incremental stability, for any $k > 0$:
\begin{equation}\label{eq:bound-on-y_i-convergence}
\begin{aligned}
    & d(\varphi(y, t_{i'} + k),y_i) = d(\varphi(\varphi(y,t_{i'}), k),\varphi(y_i, k)) \\
    & \le \sigma(d(\varphi(y, t_{i'}),y_i)) \le \epsilon.
\end{aligned}
\end{equation}
Therefore $\varphi(y,t) \to_{t \to \infty} y_i$ and this shows that $\omega(y)$ is a single equilibrium point. Using the triangle inequality twice:
\begin{equation}
\begin{aligned}
    & |d(\varphi(y,t),\varphi(x,t)) - d(P(y),P(x))| \le \\
    & d(\varphi(y,t),P(y)) + d(\varphi(x,t),P(x)) \to_{t \to \infty} 0.
\end{aligned}
\end{equation}
Hence:
\begin{equation}
    d(P(y), P(x)) = \lim_{t \to \infty} d(\varphi(y,t), \varphi(x,t))  \le \sigma(d(y,x)).
\end{equation}
Because $\sigma$ is a class $\mathcal{K}$ function, this implies $P$ is uniformly continuous with respect to $d$.
\end{proof}

\textit{Theorem \ref{thm:equivalence-limit-map-I.S}}

\begin{proof}
 \paragraph{$(1) \Rightarrow (2)$}
 Since the system is IS and $S$ is locally asymptotically stable then by Lemma \ref{lemma:local-to-global}, $S$ is globally asymptotically stable. Let $s \in S$ the IS condition implies:
 \begin{equation}
     d(\varphi(x,t), s) \le \sigma(d(x,s))
 \end{equation}
 hence $s$ is also clearly Lyapunov stable. Lastly Lemma \ref{lem:omega-limit-singleton} proves the endpoint map $P$  exists.
 \paragraph{$(2) \Rightarrow (1)$}
 Consider:
\begin{equation}
    d(x,y) = \sup_{t \ge 0} \|\varphi(x,t) - \varphi(y,t)\|.
\end{equation}
Since $P(x) = \lim_{t\to\infty} \varphi(x,t)$ exists, $d(x,y) < \infty$. We now show that $d$ is a metric. $d(\cdot,\cdot)$ is defined via the maximum separation of the flow given initial conditions $x,y$, therefore it is clear that $d(x,y) = 0$ if and only if $x = y$. For $x \neq y$ we have $d(x,y) \ge \|x-y\| > 0$. Thirdly $d(\cdot,\cdot)$ satisfies the triangle inequality via:

\begin{equation}
\begin{aligned}
    & d(x,y) = \sup_{t \ge 0} \|\varphi(x,t) - \varphi(y,t)\| \\
    & \le \sup_{t \ge 0} \|\varphi(x,t) - \varphi(z,t)\| + \sup_{t \ge 0}\|\varphi(z,t) - \varphi(y,t)\| \\
    & = d(x,z) + d(z,y).
\end{aligned}
\end{equation}

Therefore $d(\cdot,\cdot)$ is a metric. Furthermore by:

\begin{equation}
\begin{aligned}
    & d(\varphi(x,s),\varphi(y,s)) = \sup_{t \ge 0} \|\varphi(x,s+t) - \varphi(y,s+t)\| \\
    & \le \sup_{t \ge 0} \|\varphi(x,t) - \varphi(y,t)\| = d(x,y)
\end{aligned}
\end{equation}

the system is incrementally stable under $d$. It remains to show that $d$ induces the Euclidean subspace topology on $G$ and $d$ is proper. Leveraging that $d(x,y) \ge \|x-y\|$, if $d(x,y) \to 0$ then $\|x-y\| \to 0$ as well. Now suppose $\|x-y\| \to 0$ we have to show that $d(x,y) \to 0$. Without loss of generality fix $x$ and $\epsilon > 0$, by Lyapunov stability of $P(x)$ choose $\rho$ sufficiently small such that:
\begin{equation}
    \|z - P(x)\| \le \rho \Rightarrow \|\varphi(z,t) - P(x)\| \le \frac{\epsilon}{2}.
\end{equation}

Let $T$ be such that $\|\varphi(x,T) - P(x)\| \le \frac{\rho}{2}$ and which by Lyapunov stability implies $\|\varphi(x,T+s) - P(x)\| \le \frac{\epsilon}{2}
$ for all $s \ge 0$. By uniform continuity of solutions over a finite interval $[0,T]$ we can choose $y$ sufficiently close to $x$ such that both:

\begin{equation}
    \sup_{t \in [0,T]} \|\varphi(x,t) - \varphi(y,t)\| < \frac{\epsilon}{2} , \quad  \|\varphi(x,T) - \varphi(y,T)\| < \frac{\rho}{2}
\end{equation}

are satisfied. The second inequality implies via the triangle inequality that:

\begin{equation}
\begin{aligned}
    & \|\varphi(y,T) - P(x)\| \le \|\varphi(y,T) - \varphi(x,T)\| + \|\varphi(x,T) - P(x)\| \\
    & \le \rho \Rightarrow \|\varphi(y,T + s) - P(x)\| \le \frac{\epsilon}{2} \text{ for all } s \ge 0.
\end{aligned}
\end{equation}

Lastly the above also implies via the triangle inequality that:

\begin{equation}
    \|\varphi(y,T + s) - \varphi(x,T + s)\| \le \epsilon.
\end{equation}

Therefore $\|x-y\| \to 0 \Rightarrow d(x,y) \to 0$. It remains to show that the metric is proper. From the definition of $d(\cdot,\cdot)$, it holds that:
\begin{equation}
    d(x,p) \ge \|x - p\|.
\end{equation}
Therefore:
\begin{equation}
\begin{aligned}
    & \{x \;\; | \;\;  x \in G \text{ and } d(x,p) \le R\} = \overline B^{d,G}_R(p)
    \subseteq \\
    & \overline B_{R}^G(p) = \{x \;\; | \;\; x \in G \text{ and } \|x - p\| \le R\}
\end{aligned}
\end{equation}
Since $d(\cdot,\cdot)$ induces the Euclidean subspace topology, the set $\bar{B}_{R}^{d,G}(p)$ is closed relative to $G$ and boundedness follows from the above relation. Since $G$ itself is closed this proves $\bar{B}_{R}^{d}(p)$ is closed and bounded in $\mathbb{R}^n$, which shows the metric is proper. The set $S$ is stable with respect to $d$ since IS provides in the metric defined by $d$ that:
\begin{equation}
    \dist_d(\varphi(x,t),S) \le \dist_d(x,S),
\end{equation}
and it is asymptotically stable in $d$ due to,
\begin{equation}
    \dist_d(\varphi(x,t), S) \le d(\varphi(x,t),P(x)) \to 0.
\end{equation}
Lemma \ref{lem:omega-limit-singleton} proves continuity of $P$, which implies the image of the (path-)connected set $G$ is (path-)connected and since $P(G) = S$, $S$ is (path-)connected.
\end{proof}

\textit{Lemma \ref{lemma:la-Salle-I.S}}

\begin{proof}
    The existence of one bounded trajectory implies by IS that all trajectories are bounded. Furthermore the positive orbit of any initial condition $x \in G$ has compact closure $\Omega_x \subseteq G$, since the trajectory is bounded, and either $G$ is closed or by assumption every positive orbit's closure is contained in $G$. The set $\Omega_x$ is compact and forward-invariant and $\dot{V} \le 0$ on $\Omega_x$. By LaSalle's Theorem applied on $\Omega_x$, $\varphi(x,t)$ converges to the largest invariant set of $\{y \in \Omega_x: \dot{V}(y) = 0\}$, and hence to the largest invariant set of $\dot{V} = 0$. If the largest invariant set of $\dot{V} = 0$ consists of equilibria of \eqref{eq:ode}, then every (nonempty) $\omega$-limit set consists of equilibria and Lemma \ref{lem:omega-limit-singleton} proves the endpoint map exists and is continuous.
\end{proof}

\end{document}